\documentclass[11pt]{article}
\usepackage[utf8]{inputenc}
\usepackage[english]{babel}
\usepackage[T1]{fontenc}
\usepackage{float}
\usepackage{listings}
\usepackage{amsmath}
\usepackage{amsfonts}
\usepackage{amssymb}
\usepackage{amsthm}
\usepackage{enumerate}
\usepackage{anysize}
\usepackage{graphicx}
\usepackage{xcolor}
\usepackage{color}
\usepackage{fancyhdr}
\usepackage{makeidx}
\usepackage{hyperref}
\usepackage{appendix}
\usepackage{tikz}
\usepackage[matrix,arrow]{xy}
\usetikzlibrary{arrows} 
\usepackage{cite}
\usepackage{fancyhdr}
\usepackage{subcaption}

\usepackage{stackrel}
\usepackage[pagewise]{lineno}
\marginsize{2cm}{2cm}{0cm}{2cm}
\definecolor{black}{rgb}{0.0, 0.0, 0.0}

\usepackage{etoolbox}

\newtheorem{teo}{Theorem}
\newtheorem{lem}{Lemma}
\newtheorem{defi}{Definition}
\newtheorem{obs}{Remark}
\newtheorem{col}{Corollary}

\begin{document}

	\title{\textbf{Theoretical analysis for a PDE-ODE system related to a Glioblastoma tumor with vasculature}
		\thanks{The authors were supported by PGC2018-098308-B-I00 (MCI/AEI/FEDER, UE)}
	}
	\author{ A. Fernández-Romero$^{2}$, F. Guillén-González$^{2}$\footnote{ORCID: 0000-0001-5539-5888}, A. Suárez$^{1\;2}$\footnote{ORCID: 0000-0002-3701-6204}.\\ 
		\small{$^{1}$Corresponding author.}\\
	\small{$^{2}$Dpto. Ecuaciones Diferenciales y Análisis Numérico,}\\
	\small{Facultad de Matemáticas, Universidad de Sevilla. Sevilla, Spain.}\\
	\small{\texttt{afernandez61@us.es, guillen@us.es, suarez@us.es}}	
}
	\date{}
\maketitle

\section*{\centering{\textbf{Abstract}}}
In this paper we study a PDE-ODE system as a simplification of a Glioblastoma model. Mainly, we prove the existence and uniqueness of global in time classical solution using a fixed point argument. Moreover, we show some stability results of the solution depending on some conditions on the parameters.
\\
\\
\textbf{Mathematics Subject Classification.} $35\text{A}09,35\text{B}40,35\text{M}10,35\text{Q}92,47\text{J}35$\\
\textbf{Keywords:} Tumor model, Glioblastoma, PDE-ODE system, Classical solution.

\section{Introduction}
Glioblastoma (GBM) is one of the most lethal malignant brain tumor with a survival of $14.6$ months \cite{Ostrom_2014}. These include the presence of necrosis and high proliferation of cells. The magnetic resonance imaging shows a necrotic area in the center surrounded by a white ring. This ring is an indicator of areas with poor vasculature. Clinical, molecular and imaging parameters have been used to build mathematical models able to classify GBM patients in terms of survival, identify GBM subtypes, predict response to treatment, etc \cite{Gillies_2016, Ellingson_2014, Abrol_2017, Narang_2016}.
\\

Mathematical modelling has been presented as an additional tool to better understand the evolution, prediction the outcome and different therapies or classifies patients according to prognosis. Thus, the mathematical modelling of GBM is being a relatively broad topic in the community of applied mathematics. One of the reasons that explain this limitation is that either the key biological variables have not been included or real data of sufficient quality have not been used.
\\


Recently, in \cite{Victor_2020}, Molab\footnote{\href{Molab}{http://matematicas.uclm.es/molab/}} group has proposed a mathematical model of glioblastoma growth to explain the correlation between magnetic resonance images and tumor growth speed. For that, two variables are considered: tumor population and necrosis, and they quantify the tumor ring and obtain a relation with the survival. In particular, the group classifies the GMB with respect to the tumor ring and the irregular growth, see \cite{Julian_2016} and \cite{Victor_2018}. We have completed the above model including an essential variable: the vasculature since it is well-known that vasculature plays a relevant role in the tumor growth. Moreover, our model with vasculature is able to capture some phenomena that Molab group has studied about GBM such that the ring width-volume in \cite{Victor_2020,Julian_2016} and the regularity surface of the tumor in \cite{Victor_2018}.
\\

Specifically, let $\Omega\subseteq\mathbb{R}^3$ be a bounded domain and $\left(0,T_f\right)$ a time interval, with $0<T_f<+\infty$ and we analyse the following PDE-ODEs system where $T\left(t,x\right)\geq0$, $N\left(t,x\right)\geq0$ and $\Phi\left(t,x\right)\geq0$ represent the tumor density, necrotic density and vasculature concentration, respectively, at the point $x\in\Omega$ at the time $t\in\left(0,T_f\right)$. In its general form, the model is

\begin{equation}\label{probNoLineal}
\left\{\begin{array}{ccl}
\dfrac{\partial T}{\partial t} -\nabla\cdot\left(\left(\kappa_1\;P\left(\Phi,T\right)+\kappa_0\right)\nabla T\right)& = & f_1\left(T,N,\Phi\right) \\
&&\\
\dfrac{\partial N}{\partial t}& = & f_2\left(T,N,\Phi\right) \\
&&\\
\dfrac{\partial \Phi}{\partial t} & = &f_3\left(T,N,\Phi\right)\\
\end{array}\right.\end{equation}
where $0<\kappa_1,\;\kappa_0\in\mathbb{R}$ are diffusion coefficients. 
%
%
%
The nonlinear reaction functions $f_i\;:\mathbb{R}^3\rightarrow\mathbb{R}$ for $i=1,2,3$ have the following form

\begin{equation}\label{funciones}
\left\{\begin{array}{ccl}
f_1\left(T,N,\Phi\right) &=&\rho\;P\left(\Phi,T\right)T \left(1-\dfrac{T+N+\Phi}{K}\right)-\alpha\;T\;\sqrt{1-P\left(\Phi,T\right)^2}-\beta_1\; 
N\;T,\\
\\
f_2\left(T,N,\Phi\right) &=& \alpha\;T\;\sqrt{1-P\left(\Phi,T\right)^2}+\beta_1\; N\;T +\delta\;T\;\Phi +\beta_2\;N\;\Phi,\\
\\
f_3\left(T,N,\Phi\right) &=& \gamma\;\dfrac{T}{K}\;\sqrt{1-P\left(\Phi,T\right)^2}\;\Phi\left(1-\dfrac{T+N+\Phi}{K}\right)-\delta\;T\;\Phi-\beta_2\;N\;\Phi,\\
\end{array}\right.
\end{equation}
where the $\rho,\alpha,\beta,\delta,\gamma>0$ are reaction coefficients (see Table \ref{Table1}), $K>0$ is the carrying capacity coefficient and
$$P\left(\Phi, T\right)=
\dfrac{\Phi_+}{\Phi_++T_+}\;\; \text{  if  }\; \left(\Phi,T\right)\neq\left(0,0\right)$$
%
with $T_+=\max\{0,T\}$ and the same for $\Phi_+$. Notice that $P\left(\Phi,T\right)$ is the vasculature volume fraction and it has the pointwise estimate 
$$0\leq P\left(\Phi,T\right)\leq1\qquad\forall\left(T,\Phi\right)\in\mathbb{R}^2\backslash\left\{\left(0,0\right)\right\}$$
and $P\left(\Phi,T\right)=0$ for $\Phi=0$ and $P\left(\Phi,T\right)=1$ for $T=0$.
\\

In this paper, we contemplate a simplification of $\left(\ref{probNoLineal}\right)$ vanishing the nonlinear diffusion velocity $\kappa_1 P(\Phi,T)$ (i.e. a linear diffusion will be considered). Moreover, we take $\kappa_0=1$ for simplicity. The complete nonlinear diffusion problem $\left(\ref{probNoLineal}\right)$ will be studied in a forthcoming paper. Hence, we consider the following PDE-ODE system

\begin{equation}\label{probOriginal}
\left\{\begin{array}{ccl}
\dfrac{\partial T}{\partial t} -\;\Delta\;T& = & f_1\left(T,N,\Phi\right) \\
&&\\
\dfrac{\partial N}{\partial t}& = & f_2\left(T,N,\Phi\right) \\
&&\\
\dfrac{\partial \Phi}{\partial t} & = &f_3\left(T,N,\Phi\right)\\
\end{array}\right.\end{equation}
endowed with non tumor flux boundary condition
\begin{equation}\label{condifronte}
\dfrac{\partial T}{\partial n}\Bigg\vert_{\partial\Omega}=0\;\;\text{on}\;\;\left(0,T_f\right)\times\partial\Omega
\end{equation} 
where $n$ is the outward unit normal vector to $\partial\Omega$ and initial conditions

\begin{equation}\label{condinicio}
T\left(0,x\right)=T_0(x),\;N\left(0,x\right)=N_0(x),\;\Phi\left(0,x\right)=\Phi_0(x),
\;\;\;x\in\Omega.
\end{equation}

The parameters $\rho$, $\alpha$, $\beta_1$, $\beta_2$, $\gamma$, $\delta$ and $K$ are given by the following description corresponding to a result of relevant studies \cite{Alicia_2012, Alicia_2015,Klank_2018}:
\begin{table}[H]
	\centering
\begin{tabular}{c|c|c}
\textbf{Variable} & \textbf{Description} & \textbf{Value} \\
\hline
$\rho$	& Tumor proliferation rate  &  $\text{day}^{-1}$    \\
\hline
$\alpha$ & Hypoxic death rate by persistent anoxia & $cell/\text{day}$    \\
\hline
$\beta_1$	& Change rate from tumor to necrosis &   $\text{day}^{-1}$    \\
\hline
$\beta_2$	& Change rate from vasculature to necrosis &   $\text{day}^{-1}$    \\
\hline
$\gamma$	& Vasculature proliferation rate  &  $\text{day}^{-1}$    \\
\hline
$\delta$	& Vasculature destruction by tumor action  &  $\text{day}^{-1}$    \\
\hline
$K$	& Carrying capacity & $\text{cell}/\text{cm}^3$ \\
\end{tabular}
\caption{\label{Table1}Reaction coefficients.}
\end{table}

We are going to describe the biological meaning of the reaction terms:

\begin{itemize}
	\item It has been observed that tumor cells show a random movement when there is no nutrient limitation (which is modelled as a linear diffusion term). 
	
	\item Necrosis has not diffusion movement and it is experimentally known that the necrosis will grow when the tumor does.
	\item Since tumour cells and vasculature must have enough space to proliferate, two logistic growth terms have been included respectively, for tumor and vasculature.
	$$T\left(1-\dfrac{T+N+\Phi}{K}\right)\;\;\text{in}\;\; f_1\left(T,N,\Phi\right)\quad\text{and}\quad \Phi\left(1-\dfrac{T+N+\Phi}{K}\right)\;\;\text{in}\;\; f_3\left(T,N,\Phi\right)$$

	\item Since vasculature supplies nutrients and oxygenation to tumor cells, speed tumor growth depends on the amount of vasculature. Hence, the tumour growth coefficient is given by: $\rho\;P\left(\Phi,T\right)$.
	\item We consider the hypoxia term, $\alpha\;T\;\sqrt{1-P\left(\Phi,T\right)^2}$, that is, a decreasing tumor term due to lack of vasculature which is transformed into necrosis which satisfies that if the ratio of vasculature per cell is healthy, then there will be no hypoxia. Therefore, low vasculature produces more tumor destruction and high vasculature less destruction. In fact, the non-dimensional factor satisfies that
	
	$$\sqrt{1-P\left(\Phi,T\right)^2}=\left\{
	\begin{array}{lcl}
      \text{increasing to }\;1&\text{if}& \Phi\rightarrow0, \\
      \\
      \text{decreasing to }\;0&\text{if}& \Phi\rightarrow +\infty.
	\end{array}
	\right.$$
	
	Despite we have chosen this hypoxia term, other functions with the same behaviour could be contemplated.
	
	\item The vasculature growth coefficient is $\gamma\;\frac{T}{K} \sqrt{1-P\left(\Phi,T\right)^2}$. It depends on the amount of tumor and satisfies two biological conditions:

	\begin{enumerate}
	\item Vasculature can undergo growth when there is a high demand for nutrients by the tumor cells. In particular, where there is not tumor, there is not growth of vasculature.
	\item The vasculature growth term decreases with respect to the amount of vasculature.  
	\end{enumerate}
	
	\item Interaction between tumor (resp. vasculature) with necrosis produces a lost of tumor (resp. vasculature) in function of the necrosis, with the terms:  $\pm\beta_1\;T\;N$ and $\pm\beta_2\;\Phi\;N$.
	
	\item The destruction of vasculature by tumor is transformed into necrosis by the terms: $\pm\delta\;T\;\Phi$.
\end{itemize}

There is an extensive literature devoted to the study of PDE-ODE systems, see for instance \cite{Pang_2018, Anderson_2015, Bitsoumi_2017, Cruz_2018} and the references therein. As far as we know, a great quantity of works related to solve this kind of problems uses generic results of \cite{Amann_1991,Amann_1993}, see for instance \cite{Kubo_2016, Tello_2018}. 
\\

The aim of this paper is to analyse $\left(\ref{probOriginal}\right)$-$\left(\ref{condinicio}\right)$ in a theoretical way. Firstly, we show the existence and uniqueness of global in time classical solution using a fixed point argument. In fact, the fixed point operator is built by computing first the ODE system, and then the nonlinear PDE. One important difficulty here is to obtain classical regularity of solutions with respect to the spatial variable (which is a parameter for the ODE system). Secondly, we study the asymptotic behaviour of solutions of $\left(\ref{probOriginal}\right)$-$\left(\ref{condinicio}\right)$, showing three main results:
\begin{enumerate}
	\item Vasculature goes to zero as time goes to infinity pointwisely in space for any choice of parameters
	\item If the destruction of vasculature by tumor is large regarding to the vasculature growth, specifically if $\delta\geq\dfrac{\gamma}{K}$, then tumor and vasculature goes to zero in an exponential way (uniformly in space) and necrosis is uniformly bounded.
	\item If the destruction of tumor by necrosis dominates to tumor growth, specifically if $\beta_1\gg\rho$ (see hypothesis $\left(\ref{condi}\right)$ below), then tumor and vasculature go to zero in an exponential way (uniformly in space) and necrosis is uniformly bounded.
\end{enumerate}

The paper is organized as follows: In Section $\ref{prelim}$, we present preliminary results which we will use along the paper. In Section $\ref{solocion}$ we prove the existence (and uniqueness) of classical solution of $\left(\ref{probOriginal}\right)$-$\left(\ref{condinicio}\right)$. Section $\ref{comp_asin}$ is dedicated to the long time behaviour of the classical solutions 
and we show some numerical simulations according to the results proved previously. Finally, in Section $\ref{conclusion}$, we discuss our findings and summarize our main results.

\section{Preliminaries}\label{prelim}
Although $P\left(\Phi,T\right)$ is not evaluated  in $\left(0,0\right)$, we can deduce the following

\begin{lem}\label{lemaB}
	 The functions $B:\mathbb{R}^2\rightarrow\mathbb{R}$ and $D:\mathbb{R}^2\rightarrow\mathbb{R}$ given by $$B\left(\Phi,T\right)=T_+\;\sqrt{1-P\left(\Phi,T\right)^2}$$
and $$D\left(\Phi,T\right)=T_+\;\;P\left(\Phi,T\right)$$
 are well defined, continuous and globally lipschitz in $\mathbb{R}^2$.
\end{lem}

\begin{proof}
	We only show the proof for $B\left(\Phi,T\right)$ because for $D\left(\Phi,T\right)$ it is similar, even easier. Since $0\leq P\left(\Phi,T\right)\leq1$, it is clear that $B\left(\Phi,T\right)$ is well defined and continuous in $\mathbb{R}^2$ (in particular, $B(0,0)=0$). 
	 To prove the global lipschitz condition for $B$, it suffices to show that the two partial derivatives of $B\left(\Phi,T\right)$ are continuous and bounded in the subdomain $A=\left\{\left(\Phi,T\right)\in\mathbb{R}^2\;\;:\;\;\Phi,\;T>0\right\}$ (in the rest, is equal to zero). By means of direct calculations, it follows that for any $\left(\Phi,T\right)\in A$,
	
	\begin{equation}\label{derivadaBrespectoF}
\left|\dfrac{\partial B}{\partial \Phi}\right|\leq
\left|\dfrac{1}{2}\dfrac{\sqrt{T}}{\sqrt{T+2\;\Phi}}\right|\leq \dfrac{1}{2},
	\end{equation}
	and
	\begin{equation}\label{derivadaBrespectoT}
	\left|\dfrac{\partial B}{\partial T}\right|\leq
	 \left|1+\dfrac{T}{\sqrt{T}\sqrt{T+2\;\Phi}}\right|\leq 2.
	\end{equation}
\\

Hence, we deduce that $B\left(\Phi,T\right)$ is globally lipschitz in $\mathbb{R}^2$.
    
\end{proof}

%
%
%

As consequence, we get the following result
\begin{lem}\label{lemafi}
	The functions $f_i\;:\mathbb{R}^3\rightarrow\mathbb{R}$ for $i=1,2,3$ defined in $\left(\ref{funciones}\right)$ are  continuous and locally lipschitz in $\mathbb{R}^3$.
\end{lem}

\begin{proof}
	Rewriting the definition of $f_i\left(T,N,\Phi\right)$ for every $i=1,2,3$ according to the functions $B\left(\Phi,T\right)$ and $D\left(\Phi,T\right)$, it is easy to deduce that functions $f_i\left(T,N,\Phi\right)$ are continuous and their partial derivatives are bounded in compact sets of $\mathbb{R}^3$ for every $i=1,2,3$, because they are products and sums of the globally lipschitz functions $B\left(\Phi,T\right)$ and $D\left(\Phi,T\right)$ and polynomials in $(T,N,\Phi)$.
\end{proof}
	
In order to obtain some regularity result, we need to define the following spaces for $p>3$:	

$$W^{2-2/p,p}_n\left(\Omega\right)=\left\{u\in W^{2-2/p,p}\left(\Omega\right)\;:\;\frac{\partial u}{\partial n}=0\;\;\text{on}\;\;\partial\Omega \right\},$$ 

$$V_p=\left\{
\begin{array}{c}

u\in L^p\left(0,T_f;W^{2,p}\left(\Omega\right)\right)\cap \mathcal{C}^0\left(\left[0,T_f\right];W_{n}^{2-2/p,p}\left(\Omega\right)\right)\\
\\
\text{and}\;u_t \in L^p\left(0,T_f;L^p\left(\Omega\right)\right) 
\end{array}\right\}$$
with the norm,
$$\| u\|_{V_p}:=\| u\|_{\mathcal{C}^0([0,T_f];W_{n}^{2-2/p,p}\left(\Omega\right))}+\|\partial_t u\|_{L^p\left(0,T_f;L^p\left(\Omega\right)\right)}+\| u\|_{L^p\left(0,T_f;W^{2,p}\left(\Omega\right)\right)}.$$

The following result follows by \cite[p. 344]{Feireisl_2009}
\begin{lem}\label{w2p}
	Assume $\Omega\in\mathcal{C}^2$, let $p>3$, $u_0\in W^{2-2/p,p}_n\left(\Omega\right)$ and $g\in L^p\left(0,T_f;L^p\left(\Omega\right)\right)$. Then, the problem
	
	$$\left\{\begin{array}{rcl}
	\partial_t u- \Delta u &=& g\qquad\text{in}\;\; \left(0,T_f\right)\times\Omega,\\
	\\
	u(0,\cdot)&=&u_0\qquad\text{in}\;\;\Omega,\\
	\\
	\dfrac{\partial u}{\partial n}&=&0\qquad\text{on}\;\;\left(0,T_f\right)\times\partial\Omega,
	\end{array}\right.$$
	admits a unique solution $u\in V_p$. Moreover, there exists a positive constant $C:=C\left(p,\Omega,T_f\right)$ such that
	$$\| u\|_{V_p}\leq C\left(\| g\|_{L^p\left(0,T_f;L^p\left(\Omega\right)\right)},\;\| u_0\|_{W_{n}^{2-2/p,p}\left(\Omega\right)}\right).$$
\end{lem}

	It will be necessary to obtain existence and uniqueness of global in time classical solution for an ordinary differential system depending on parameters. The first result is a classical extension result while the second one provides us the continuous dependence of the solutions of an ODE system with respect to parameters and initial conditions, see \cite{Coddington_1955} for instance. 
\begin{lem}[Continuous extension]\label{lemaExt}
	Let $g\in\mathcal{C}^0\left(\overline{\Omega}\right)$ with $\Omega\subseteq\mathbb{R}^d$ an open bounded set of class $\mathcal{C}^0$ and $d\in\mathbb{N}$. Then, there exists an extension $Ext\left(g\right)\in\mathcal{C}^0\left(\mathbb{R}^d\right)$ such that $Ext\left(g\right)\big\vert_{\overline{\Omega}}=g$.
\end{lem}

\begin{teo}[Continuous dependence of ODEs with respect to parameters and initial data]\label{parameter}
	Let $U\subset\mathbb{R}\times\mathbb{R}^N\times\mathbb{R}^M$ an open set and $F\;:\;U\rightarrow\mathbb{R}^N$ a continuous map such that, for any parameter $\lambda\in\mathbb{R}^M$ and for any initial data $y_0\left(\lambda\right)\in\mathbb{R}^N$ such that $\left(0,y_0\left(\lambda\right),\lambda\right)\in U$, the Cauchy's problem
	
	$$\left\{ \begin{array}{l}
	
	y'\left(t\right)=F\left(t,y,\lambda\right)\\
	\\
	y\left(0\right)=y_0\left(\lambda\right)
	\end{array}\right.$$
	has a unique maximal solution $\phi\left(\cdot;y_0\left(\lambda\right),\lambda\right)\;:\;I_{\left(y_0\left(\lambda\right),\lambda\right)}\rightarrow\mathbb{R}^N$ being $I_{(y_0(\lambda),\lambda)}$ an open interval. Then,
	$$\Theta=\left\{\left(t;y_0\left(\lambda\right),\lambda\right)\in\mathbb{R}\times\mathbb{R}^N\times\mathbb{R}^M\;:\:\left(t,y_0\left(\lambda\right),\lambda\right)\in U\text{  and  }t\in I_{\left(y_0\left(\lambda\right),\lambda\right)}\right\}$$
	is an open set and the map $\phi\left(\cdot;\cdot,\cdot\right)$ is continuous from $\Theta$ to $\mathbb{R}^N$.
\end{teo}

Finally, we will use this classical fixed point theorem.

\begin{teo}[\textit{Leray-Schauder's} theorem] \label{Leray}
	Let $V$ a Banach space, $\lambda\in\left[0,1\right]$ and $\mathcal{R}:V\rightarrow V$ a continuous and compact map such that for every $v\in V$ with $v=\lambda\;\mathcal{R}(v)$, it holds $\|v\|_V\leq \mathcal{C}$ with $\mathcal{C}>0$ independent of $\lambda\in\left[0,1\right]$. 
	Then, there exists a fixed point $v$ of $\mathcal{R}$. 
\end{teo}

\section{Existence and uniqueness of Classical Solution of Problem $\boldsymbol{\left(\ref{probOriginal}\right)$-$\left(\ref{condinicio}\right)}$}\label{solocion}

First of all, by biological considerations, we assume along the paper the following assumption on the initial data
\begin{equation}\label{hipotesis0}
0\leq T_0(x),N_0(x),\Phi_0(x)\leq K\; \text{  in  } \;\Omega.
\end{equation}

Now, we define the concept of classical solution of $\left(\ref{probOriginal}\right)$-$\left(\ref{condinicio}\right)$.

\begin{defi}\label{Defi1}(Classical solution of $\left(\ref{probOriginal}\right)$-$\left(\ref{condinicio}\right)$)
%
Given $T_0\in W^{2-2/p,p}\left(\Omega\right)$ for some $p>3$ and $N_0,\;\Phi_0\in\mathcal{C}^0\left(\overline{\Omega}\right)$, then $\left(T,N,\Phi\right)$ is called a classical solution of $\left(\ref{probOriginal}\right)$-$\left(\ref{condinicio}\right)$ if: 

\begin{enumerate}[i)]
	\item $T\in V_p$, $N,\Phi\in{C}^1\left(\left[0,T_f\right];\mathcal{C}^0\left(\overline{\Omega}\right)\right)$,

		\item \begin{itemize}
			\item $\displaystyle T_t-\Delta\;T= f_1\left(T,N,\Phi\right)$ a.e. in $\left(0,T_f\right)\times\Omega$,
			\vspace{0.5cm}
			\item $\left(\begin{array}{c}
				\dfrac{\partial N}{\partial t}\\
				\\
				\dfrac{\partial \Phi}{\partial t}
			\end{array}\right) = \left(\begin{array}{c}
				f_2\left(T,N,\Phi\right)\\
				\\
				\\
				f_3\left(T,N,\Phi\right)
			\end{array}\right)\quad\forall\;\left(t,x\right)\in\left[0,T_f\right]\times\overline{\Omega}$ .
		
		\end{itemize}
	
	\item $\left(T,N,\Phi\right)$ satisfies the boundary and the initial conditions given in $\left(\ref{condifronte}\right)$ and $\left(\ref{condinicio}\right)$ respectively.
	\end{enumerate}
     
\end{defi}

\begin{teo}\label{unicidad}
	If there exists a classical solution of $\left(\ref{probOriginal}\right)$-$\left(\ref{condinicio}\right)$, then, it is unique.
\end{teo}
\begin{proof}
	Let $\left(T_1,N_1,\Phi_1\right)$ and $\left(T_2,N_2,\Phi_2\right)$ two possible classical solutions  of $\left(\ref{probOriginal}\right)$-$\left(\ref{condinicio}\right)$. Since the two solutions are classical solutions, fixed a final time $0<T_f<+\infty$, we have that $\left(T_1,N_1,\Phi_1\right)$ and $\left(T_2,N_2,\Phi_2\right)$ are bounded pointwise. Then, the graphs $\left(T_i\left(t,x\right),N_i\left(t,x\right),\Phi_i\left(t,x\right)\right)$ for any $(t,x)\in [0,T_f]\times \overline\Omega$ are bounds for $i = 1, 2$ and therefore the union of both graphs is contained in a compact $\cal K$ of $\mathbb{R}^3$. We consider the problem which satisfies the difference $T=T_1-T_2$, $N=N_1-N_2$, $\Phi=\Phi_1-\Phi_2$,

	\begin{equation}\label{Unicidad}
	\left\{\begin{array}{ccl}
	\dfrac{\partial T}{\partial t}-\Delta\;T& = & f_1\left(T_1,N_1,\Phi_1\right)- f_1\left(T_2,N_2,\Phi_2\right) \\
	&&\\
	\dfrac{\partial N}{\partial t}& = & f_2\left(T_1,N_1,\Phi_1\right)- f_2\left(T_2,N_2,\Phi_2\right) \\
	&&\\
	\dfrac{\partial \Phi}{\partial t}& = & f_3\left(T_1,N_1,\Phi_1\right)- f_3\left(T_2,N_2,\Phi_2\right)
	\\
	
	\end{array}\right.\end{equation}
	with non-flux boundary condition and zero initial data
	
	$$
	\dfrac{\partial T}{\partial \text{n}}\Bigg\vert_{\partial\Omega}=0,\label{condifronteraunicidad}\\
	\nonumber\quad
	T\Big\vert_{t=0}=N\Big\vert_{t=0}=\Phi\Big\vert_{t=0}=0\label{condiniciounicidad}.
	$$
	
	It is sufficient to prove that $\left(T,N,\Phi\right)\equiv\left(0,0,0\right)$.
	Multiplying the first equation of $\left(\ref{Unicidad}\right)$ by $T$ and integrating in $\Omega$, we obtain
	
	\begin{equation}\label{uniT}
	\begin{array}{c}
	\displaystyle\dfrac{1}{2}\;\dfrac{d}{dt}\int_{\Omega}T^2\;dx+\int_{\Omega}\Big|\nabla T\Big|^2\;dx=\int_{\Omega}\Big| \left(f_1\left(T_1,N_1,\Phi_1\right)- f_1\left(T_2,N_2,\Phi_2\right)\right)\;T\Big|\;dx\\
	\\
	\displaystyle\leq C_1\left(\int_{\Omega}\left(T^2+\big|N\big|\;\big|T\big|+\big|\Phi\big|\;\big|T\big|\right)\;dx\right)\leq C_1\left(\int_{\Omega}\left(T^2+N^2+\Phi^2\right)\;dx\right)
	\end{array}
	\end{equation}
	because $f_1\left(T,N,\Phi\right)$ is  locally lipschitz in $\mathbb{R}^3$ and $\left(T_i,N_i,\Phi_i\right)\left(t,x\right)$ is bounded in $\mathbb{R}^3$ for $i=1,2$. We repeat the same argument for the second and third equations, multiplying by $N$ and $\Phi$, respectively. 
	\\

	We conclude that, 
	
	\begin{equation}\label{uniSuma}
	\dfrac{1}{2}\;\dfrac{d}{dt}\int_{\Omega}\left(T^2+N^2+\Phi^2\right)dx+\int_{\Omega}\Big|\nabla T\Big|^2\;dx\leq C\int_{\Omega}\left(T^2+N^2+\Phi^2\right).
	\end{equation}

	Consequently, $T,\;N,\;\Phi\equiv0$.

\end{proof}

In order to obtain existence of solution for the system $\left(\ref{probOriginal}\right)$-$\left(\ref{condinicio}\right)$, we define the following truncated system of $\left(\ref{probOriginal}\right)$:

\begin{equation}\label{problin}
\left\{\begin{array}{ccl}
\dfrac{\partial T}{\partial t} -\Delta\;T& = & f_1\left(T_+,N_+,\Phi_+\right),\\
&&\\
\dfrac{\partial N}{\partial t} & = &  f_2\left(T_+^{K},N_+,\Phi_+\right),\\
&&\\
\dfrac{\partial \Phi}{\partial t} & = & f_3\left(T_+^{K},N_+,\Phi_+\right),\\
\end{array}\right.
\end{equation}
\\
endowed with the boundary and initial conditions given in $\left(\ref{condifronte}\right)$ and $\left(\ref{condinicio}\right)$ where $T_+^{K}=\min\left\{K,\max\left\{T,0\right\}\right\}$.
\\
%

Once we prove the existence of classical solution of the problem $\left(\ref{problin}\right)$ and its positivity, we will deduce in fact that this solution is also a classical solution of $\left(\ref{probOriginal}\right)$-$\left(\ref{condinicio}\right)$. 
\\


Before studying the existence of classical solution of $\left(\ref{problin}\right)$, we prove a priori estimates for any possible classical solution. 
\\

\begin{lem}[Pointwise a priori estimates.]\label{estimaciones}
	Under assumptions of Definition $\ref{Defi1}$, any classical solution $\left(T,N,\Phi\right)$ of the truncated system $\left(\ref{problin}\right)$ with initial data verifying $\left(\ref{hipotesis0}\right)$ satisfies the following pointwise bounds
	\begin{equation}\label{cotas}
	\left\{\begin{array}{lr}\
	0\leq T\leq K,&\text{ a.e. }\left(t,x\right)\in\left(0,T_f\right)\times\Omega,\\
	\\
	0\leq N\leq C\left(T_{f}\right),&\forall\left(t,x\right)\in\left[0,T_f\right]\times\overline{\Omega},\\
	\\
	0\leq \Phi\leq K,&\forall\left(t,x\right)\in\left[0,T_f\right]\times\overline{\Omega},
	\end{array} \right.\end{equation}
	where $C\left(T_f\right)$ is a positive constant depending exponentially on the final time $T_f$, which we will define below in $\left(\ref{*}\right)$.
\end{lem}

\begin{proof}
	
Let $\left(T,N,\Phi\right)$ be a classical solution of $\left(\ref{problin}\right)$.
%
%
%
%
%
Multiplying the first equation of $\left(\ref{problin}\right)$ by $T_-=\min\left\{T,0\right\}$ and integrating in $\Omega$, if we rewrite $f_1(T_+,N_+,\Phi_+)=T_+\; \widetilde{f}_1(T_+,N_+,\Phi_+)$, we get

$$\dfrac{1}{2}\dfrac{d}{dt}\int_{\Omega}(T_-)^2\;dx+\int_{\Omega}\mid\nabla T_-\mid^2\;dx=
\int_{\Omega}T_-\;T_+\;\widetilde{f}_1\left(T_+,N_+,\Phi_+\right)\;dx=0,\;\;\text{ a.e. }t\in\left(0,T_f\right).$$
%
%
%

Hence, since $T_-\left(0,x\right)=0$, we get $T_{-}\left(t,x\right)=0$ a.e. $\left(t,x\right)\in\left(0,T_f\right)\times\Omega$.
%
We can repeat the same argument for the other two equations of $\left(\ref{problin}\right)$, using that 

$$\Phi_-\;f_3\left(T_+^{K},N_+,\Phi_+\right)=0\;\;\text{and}\;\;N_-\;f_2\left(T_+^{K},N_+,\Phi_+\right)\leq0.$$ 

To obtain the upper bounds of $\left(\ref{cotas}\right)$, we multiply the first equation of $\left(\ref{problin}\right)$ by $\left(T-K\right)_+=\max\left\{0,T-K\right\}$ and integrate in $\Omega$,
\\
$$\dfrac{1}{2}\dfrac{d}{dt}\int_{\Omega}\left(\left(T-K\right)_+\right)^2\;dx+\int_{\Omega}\mid\nabla \left(T-K\right)_+\mid^2\;dx=$$
$$=\int_{\Omega}f_1\left(T_+,N_+,\Phi_+\right)\left(T-K\right)_+\;dx\leq0,\;\;\text{ a.e.}\;\;t\in\left(0,T_f\right)$$
\\
%
%
%
%
%
%
where in the last inequality we have used $f_1\left(T_+,N_+,\Phi_+\right)\leq \rho\;T_+\left(1-\dfrac{T_+}{K}\right)$. 
\\

Hence, since $\left(T\left(0,x\right)-K\right)_+=0$, then $\left(T\left(t,x\right)-K\right)_+=0$ a.e. $\left(t,x\right)\in\left(0,T_f\right)\times\Omega$. We repeat the same argument for the third equation of $\left(\ref{problin}\right)$ using that $\left(\Phi-K\right)_+\;f_3\left(T_+^{K},N_+\Phi_+\right)\leq0$.
\\
%
%
%
%
%

Finally, given a fixed final time $T_f>0$, 
for any $t\leq T_f$ and $x\in\overline{\Omega}$, we have

$$\dfrac{\partial N}{\partial t} =\alpha\;B\left(\Phi_+,T_+^{K}\right)+\delta\;T_+^{K}\;\Phi_+ +N\left(\beta_1\;T_+^{K}+\beta_2\;\Phi_+\right)\leq \mathcal{C}_1+\mathcal{C}_2\;N$$
where $\mathcal{C}_1$ and $\mathcal{C}_2$ depend on $\alpha$, $\beta_1$, $\beta_2$, $\delta$ and $K$. Hence,

\begin{equation}\label{*}
N\left(t,x\right)\leq \dfrac{\mathcal{C}_1}{\mathcal{C}_2}\left(e^{\mathcal{C}_2\;t}-1\right)+e^{\mathcal{C}_2\;t}\;N_0\left(x\right)\leq C\left(T_f\right)=e^{\mathcal{C}_2\;T_f}\left(\dfrac{\mathcal{C}_1}{\mathcal{C}_2}+K\right)=\mathcal{C}\left(T_f\right).
\end{equation}

In particular, $C\left(T_f\right)>0$ is an upper bound with an exponential growth depending on the final time $T_f$.
\end{proof}

By Lemma \ref{estimaciones}, 
we deduce that if $\left(T,N,\Phi\right)$ is a classical solution of $\left(\ref{problin}\right)$ then $T_+^{K}=T$, $N_+=N$ and $\Phi_+=\Phi$ and $f_i\left(T_+^{K},N_+,\Phi_+\right)=f_i\left(T,N,\Phi\right)$ for $i=1,2,3$. Hence, we obtain the following crucial corollary
\begin{col}\label{Equiv}
	Under hypotheses of Lemma $\ref{estimaciones}$, if $\left(T,N,\Phi\right)$ is a classical solution of the truncated problem $\left(\ref{problin}\right)$, then $\left(T,N,\Phi\right)$ is also a classical solution of the non truncated problem $\left(\ref{probOriginal}\right)$-$\left(\ref{condinicio}\right)$ and $\left(T,N,\Phi\right)$ satisfies the pointwise bounds $\left(\ref{cotas}\right)$.
\end{col}

\begin{teo}[Existence of classical solution of $\left(\ref{problin}\right)$]
Let $\Omega\subseteq\mathbb{R}^3$ be a bounded domain of class $\mathcal{C}^2$ and $\left(0,T_f\right)$ a time interval, with $0<T_f<+\infty$ and let $T_0\in W_{n}^{2-2/p,p}\left(\Omega\right)$ for some $p>3$ and $N_0,\Phi_0\in\mathcal{C}^0\left(\overline{\Omega}\right)$ satisfying $\left(\ref{hipotesis0}\right)$. Then, there exists a unique classical solution $\left(T,N,\Phi\right)$ of system $\left(\ref{problin}\right)$ in the sense of Definition $\ref{Defi1}$. Moreover, $\left(T,N,\Phi\right)$ satisfies estimates $\left(\ref{cotas}\right)$.
\end{teo}

%
%
	
	\begin{obs}
		In the revision process, one of the referees pointed out that the proof of the existence and uniqueness of the global classical solution could be deduced from the Rothe's book \cite{Rothe_1984}. In fact, the part II of \cite{Rothe_1984} is devoted to degenerate parabolic systems with linear diffusion, where some variables have zero diffusion coefficient, remaining a mixed PDE-ODE system as $\left(\ref{probOriginal}\right)$. The argument developed in \cite{Rothe_1984} is completely different to ours made in this paper. In fact, in \cite{Rothe_1984} the existence and uniqueness of a mild solution (satisfying an integral system) is proved in three steps, first local existence via a contractive map, second the length of the local time existence is a bounded from below and third proving an extensibility result. By the contrary, here we will prove the existence of global in time solution directly by applying the Leray-Schauder fixed-point Theorem. 
		
	\end{obs}
\begin{proof}
The proof splits in several steps: 


\paragraph{Step $\boldsymbol{1}$}\hspace{-0.38cm}\textbf{.}
\\
\\
We define the map

$$
\begin{array}{cccccc}
\mathbf{R}:& \mathcal{C}^0\left(\left[0,T_f\right];\mathcal{C}^0\left(\overline{\Omega}\right)\right)&\stackbin{R_1}{\rightarrow}&\left(\mathcal{C}^1\left(\left[0,T_f\right];\mathcal{C}^0\left(\overline{\Omega}\right)\right)\right)^2&\stackbin{R_2}{\rightarrow}& \mathcal{C}^0\left(\left[0,T_f\right];\mathcal{C}^0\left(\overline{\Omega}\right)\right)\\
\\
&\widetilde{T}&&\left(N,\Phi\right)&&T
\end{array}
$$
where $R_1(\widetilde{T}):=\left(N,\Phi\right)$ is the solution of the ordinary differential problem 
%
%
%

\begin{equation}\label{SDOlin}
\left\{\begin{array}{c}
\left(\begin{array}{c}
\dfrac{\partial N}{\partial t}\\
\\
\dfrac{\partial \Phi}{\partial t}
\end{array}\right) = \left(\begin{array}{c}
f_2\left(\widetilde{T}_+^{K},N_+,\Phi_+\right)\\
\\
f_3\left(\widetilde{T}_+^{K},N_+,\Phi_+\right)
\end{array}\right)
\\
\\
\left(\begin{array}{c}
N\left(0,x\right)
\\
\\
\Phi\left(0,x\right)
\end{array}\right)
=
\left(\begin{array}{c}
N_0\left(x\right)\\
\\
\Phi_0\left(x\right)
\end{array}\right)
\end{array}
\right.
\end{equation}
%
%
%
and $R_2\left(N,\Phi\right)=:T$ is the solution of the nonlinear parabolic problem,

\begin{equation}\label{EDPlin}
\left\{\begin{array}{l}
T_t-\Delta\;T = f_1\left(T_+,N_+,\Phi_+\right),\\
\\
\dfrac{\partial T}{\partial \text{n}}\Bigg\vert_{\partial\Omega}=0,\\
\\
T\left(0,\cdot\right)=T_0(x).
\end{array}\right.
\end{equation}


%
%
%

\paragraph{Step $\boldsymbol{2}$.}

%
\begin{lem}\label{R1welldefi_conti}
	The map $R_1:\mathcal{C}^0\left(\left[0,T_f\right];\mathcal{C}^0\left(\overline{\Omega}\right)\right)\rightarrow\left(\mathcal{C}^1\left(\left[0,T_f\right];\mathcal{C}^0\left(\overline{\Omega}\right)\right)\right)^2$ is well defined and it is continuous.
\end{lem}
\begin{proof}
	
Step $1$: $R_1$ is well defined. Observe that to obtain the solution $\left(N,\Phi\right)$ of $\left(\ref{SDOlin}\right)$, we have to solve an ordinary differential system which depends on the parameter $x\in\overline{\Omega}$, appearing in the ODE system via the function $\widetilde{T}_+^K(t,x)$ and on the initial data $\left(N_0\left(x\right),\Phi_0\left(x\right)\right)$. 
\\

We are going to define time and space extensions, respectively. First, we define the constant time extension as follows

$$\begin{array}{rcl}
Ext_t:\mathcal{C}^0\left(\left[0,T_f\right]\right)&\rightarrow&\mathcal{C}^0\left(\mathbb{R}\right)
\\
f&\mapsto&
Ext_t\left(f\right)=\left\{
\begin{array}{l}
f\left(0\right)\quad t\leq0,\\
\\
f\left(t\right)\quad 0\leq t\leq T_f,\\
\\
f\left(T_f\right)\quad t\geq T_f.
\end{array}
\right.
\end{array}$$

For the space extension, we use Lemma $\ref{lemaExt}$
$$\begin{array}{rcl}
Ext_x:\mathcal{C}^0\left(\overline{\Omega}\right)&\rightarrow&\mathcal{C}^0\left(\mathbb{R}^3\right)
\\
\\
f&\mapsto&Ext_x\left(f\right).
\end{array}$$

Finally, we consider the global extension 

$$\begin{array}{rcl}
Ext:\mathcal{C}^0\left(\left[0,T_f\right];\mathcal{C}^0\left(\overline{\Omega}\right)\right)&\rightarrow&\mathcal{C}^0\left(\mathbb{R};\mathcal{C}^0\left(\mathbb{R}^3\right)\right)
\\
\\
f&\mapsto&Ext\left(f\right):=\left(Ext_t\circ Ext_x\right)\left(f\right).
\end{array}$$

Hence, we can rewrite $\left(\ref{SDOlin}\right)$ defined in open sets as


\begin{equation}\label{SDOlin2}
\left\{ \begin{array}{l}

y'\left(t\right)=F\left(t,y,x\right)\in\mathbb{R}^2\;\;\text{  for  }\left(t,y,x\right)\in\mathbb{R}\times\mathbb{R}^2\times\mathbb{R}^3\\
\\
y\left(0\right)=y_0\left(x\right)\in\mathbb{R}^2
\end{array}\right.
\end{equation} 
\\
where we denote $y=\left(N,\Phi\right)$ and
\\
\begin{equation}\label{F1}
F\left(t,y,x\right)=\left(\begin{array}{c}
f_2\left(\left(Ext\left(\widetilde{T}\left(t,x\right)\right)\right)_+^{K},N_+,\Phi_+\right)\\
\\
f_3\left(\left(Ext\left(\widetilde{T}\left(t,x\right)\right)\right)_+^{K},N_+,\Phi_+\right)
\end{array}\right) ,
\end{equation}
\\
\begin{equation}\label{condinicialF1}
y_0\left(x\right)=\left(\begin{array}{c}
Ext_x\left(N_0\left(x\right)\right)\\
\\
Ext_x\left(\Phi_0\left(x\right)\right)
\end{array}\right).
\end{equation}
\\

Since $$0\leq \left(Ext\left(\widetilde{T}\left(t,x\right)\right)\right)_+^{K}\leq K \quad\forall\left(t,x\right)\in\left[0,T_f\right]\times\overline{\Omega}$$ and $$0\leq Ext_x\left(N_0\left(x\right)\right),Ext_x\left(\Phi_0\left(x\right)\right)\leq K \quad\forall x\in\overline{\Omega},$$ we can argue similarly to Lemma $\ref{estimaciones}$ to conclude that the solution of $\left(\ref{SDOlin2}\right)$ satisfies that $0\leq\Phi\left(t,x\right)\leq K$ and $0\leq N\left(t,x\right)\leq C\left(T_f\right)$ for all $\left(t,x\right)\in\left[0,T_f\right]\times\overline{\Omega}$.
\\

Then, by Lemmas $\ref{lemaB}$ and $\ref{lemafi}$ and definition of $F$, we have that $F\left(t,y,x\right)$ is continuous in $\mathbb{R}\times\mathbb{R}^2\times\mathbb{R}^3$ and locally lipschitz with respect to $y\in\mathbb{R}^2$. Hence for each $x\in\overline{\Omega}$ we can apply the Picard's theorem to obtain a local in time unique solution $y\left(\cdot,x\right)$ of $\left(\ref{SDOlin2}\right)$. Moreover, since we know that the solution of $\left(\ref{SDOlin2}\right)$ is bounded for all $t\in\left[0,T_f\right]$, the solution can be extended to $\left[0,T_f\right]$ for each $x\in\overline{\Omega}$.
\\


Now, we can apply Theorem $\ref{parameter}$, with $U=\mathbb{R}\times\mathbb{R}^2\times\mathbb{R}^3$, $\lambda=x\in\mathbb{R}^3$ and $y_0\left(x\right)$ defined in $\left(\ref{condinicialF1}\right)$ to the Cauchy's problem $\left(\ref{SDOlin2}\right)$. Thus, we have that for each $y_0=y_0(x)\in\mathbb{R}^2$ defined in $\left(\ref{condinicialF1}\right)$ such that $0\leq Ext_x(N_0(x)), Ext_x(\Phi_0(x))\leq K$ in $\mathbb{R}^3$, the interval $\left[0,T_f\right]\subseteq I_{\left(Ext_x(N_0(x)), Ext_x(\Phi_0(x))\right)}$ and hence, the set

$$\widetilde{\Theta}=\left\{ \begin{array}{c}
\left(t,\left(Ext_x\left(N_0\left(x\right)\right),Ext_x\left(\Phi_0\left(x\right)\right)\right),x\right)\in \mathbb{R}\times\mathbb{R}^2\times\mathbb{R}^3\;:\;
t\in I_{\left(0,\left(Ext_x(N_0(x)), Ext_x(\Phi_0(x))\right)\right)}
\end{array} \right\}$$
is an open set of $\mathbb{R}^6$ and the map $y=y\left(t;\left(Ext_x\left(N_0\left(x\right)\right),Ext_x\left(\Phi_0\left(x\right)\right)\right),x\right)$ is continuous from $\widetilde{\Theta}$ to $\mathbb{R}^2$.
\\

In conclusion, given $N_0,\;\Phi_0\in\mathcal{C}^0\left(\overline{\Omega}\right)$ such that $0\leq N_0,\Phi_0\leq K$ in $\overline{\Omega}$, there exists a solution $y=y\left(t;\left(Ext_x\left(N_0\left(x\right)\right),Ext_x\left(\Phi_0\left(x\right)\right)\right),x\right)$ of $\left(\ref{SDOlin2}\right)$ whose restriction to $[0,T_f]\times \overline\Omega$ $$\left(N,\Phi\right)\left(t,x\right)=y\left(t;\left(Ext_x\left(N_0\left(x\right)\right),Ext_x\left(\Phi_0\left(x\right)\right)\right),x\right)\quad\forall \left(t,x\right)\in\left[0,T_f\right]\times\overline{\Omega}$$ 
satisfies that
$$\left(N,\Phi\right)\in\left(\mathcal{C}^1\left(\left[0,T_f\right];\mathcal{C}^0\left(\overline{\Omega}\right)\right)\right)^2$$
and it is the unique solution of $\left(\ref{SDOlin}\right)$. 
\\

Step $2$: $R_1$ is continuous. Take $\widetilde{T}_n\rightarrow \widetilde{T}$ in $\mathcal{C}^0\left(\left[0,T_f\right];\mathcal{C}^0\left(\overline{\Omega}\right)\right)$. 
We use the same vectorial notation as before and we consider the following integral formulation of $\left(\ref{SDOlin}\right)$, 

$$y\left(t;y_0\left(x\right)x\right)=y_0\left(x\right)+\int_{0}^{t}\widetilde{F}\left(s,y\left(s,x\right),x\right)\;ds$$
where in this case $y=\left(N,\Phi\right)$ and

\begin{equation}\label{F1gorro}
\widetilde{F}\left(t,y,x\right)=\left(\begin{array}{c}
f_2\left(\widetilde{T}_+^{K}\left(t,x\right),N,\Phi\right)\\
\\
f_3\left(\widetilde{T}_+^{K}\left(t,x\right),N,\Phi\right)
\end{array}\right).
\end{equation}

%
%

Now, we take $R_1\left(\widetilde{T}_n\right)=y_n$ and $R_1\left(\widetilde{T}\right)=y$ the solutions of $\left(\ref{SDOlin}\right)$ associated to $\widetilde{T}_n$ and $\widetilde{T}$, respectively. Thus, denoting $y\left(t,\cdot\right)=y\left(t;y_0\left(\cdot\right),\cdot\right)$, we get

$$\Big\| y_n\left(t\right)- y\left(t\right)\Big\|_{\left(\mathcal{C}^0\left(\overline{\Omega}\right)\right)^{2}}=\displaystyle\Bigg\|\int_{0}^{t}\widetilde{F}\left(s,y_n\left(s\right),x\right)-\widetilde{F}\left(s,y\left(s\right),x\right)\;ds\Bigg\|_{\left(\mathcal{C}^0\left(\overline{\Omega}\right)\right)^{2}}
.$$
By Lemma $\ref{lemafi}$ and the form of $\left(\ref{F1gorro}\right)$, we deduce that $\widetilde{F}\left(t,y,x\right)$ is locally lipschitz in $\mathbb{R}\times\mathbb{R}^2\times\mathbb{R}^3$ with respect to $\left(t,y,x\right)$ . Moreover, $y_n$ and $y$ are bounded in $\mathcal{C}^0\left(\left[0,T_f\right];\mathcal{C}^0\left(\overline{\Omega}\right)\right)$, then, we have that

$$\Big\|y_n\left(t\right)-y\left(t\right)\Big\|_{\mathcal{C}^0\left(\overline{\Omega}\right)^2}\leq \mathcal{C}\;\displaystyle\int_{0}^{t}\left(\Big\|\left(y_n-y\right)\left(s\right)\Big\|_{\left(\mathcal{C}^0\left(\overline{\Omega}\right)\right)^2}+\Big\|\left(\left(\widetilde{T_n}\right)_+^{K}-\widetilde{T}_+^{K}\right)\left(s\right)\Big\|_{\mathcal{C}^0\left(\overline{\Omega}\right)}\right)\;ds.$$

Applying Gronwall's lemma, we deduce

\begin{equation}\label{R_1_cont}
\Big\|\left(y_n-y\right)\left(t\right)\Big\|_{\mathcal{C}^0\left(\overline{\Omega}\right)^2}\leq \displaystyle\mathcal{C}\;e^{\mathcal{C}\;t}\left(\int_{0}^{T_f}\Big\|\left(\left(\widetilde{T_n}\right)_+^{K}-\widetilde{T}_+^{K}\right)\left(s\right)\Big\|_{\mathcal{C}^0\left(\overline{\Omega}\right)}\;ds\right).
\end{equation}
\\

Now, in $\left(\ref{R_1_cont}\right)$ we take maximum in $t\in\left[0,T_f\right]$  in the left side and we bound in the right side. Thus,

$$\Big\|\left(y_n-y\right)\Big\|_{\mathcal{C}^0\left(\left[0,T_f\right];\mathcal{C}^0\left(\overline{\Omega}\right)\right)^2}\leq \displaystyle\mathcal{C}\;e^{\mathcal{C}\;T_f}\Big\|\left(\left(\widetilde{T_n}\right)_+-\widetilde{T}_+\right)\Big\|_{\\\mathcal{C}^0\left(\left[0,T_f\right];\mathcal{C}^0\left(\overline{\Omega}\right)\right)}\stackbin[n\rightarrow\infty]{}{\longrightarrow}0.$$

Hence, we obtain that $y_n\rightarrow y$ in $\left(\mathcal{C}^0\left(\left[0,T_f\right];\mathcal{C}^0\left(\overline{\Omega}\right)\right)\right)^2$.
\\

Moreover, it follows $$\widetilde{F}\left(y_n\left(t,x\right),t,x\right)\stackbin[n\rightarrow\infty]{}{\longrightarrow}\widetilde{F}\left(y\left(t,x\right),t,x\right)\;\;\text{ in}\;\; \left(\mathcal{C}^0\left(\left[0,T_f\right];\mathcal{C}^0\left(\overline{\Omega}\right)\right)\right)^2$$
whence we deduce that 
$$\partial_t\;y_n\left(t,x\right)=\widetilde{F}\left(y_n\left(t,x\right),t,x\right)\stackbin[n\rightarrow\infty]{}{\longrightarrow}\widetilde{F}\left(y\left(t,x\right),t,x\right)= \partial_t\;y\left(t,x\right) \text{ in } \left(\mathcal{C}^0\left(\left[0,T_f\right];\mathcal{C}^0\left(\overline{\Omega}\right)\right)\right)^2.$$
%

Hence, we get that $R_1$ is continuous from $\mathcal{C}^0\left(\left[0,T_f\right];\mathcal{C}^0\left(\overline{\Omega}\right)\right)$ to $\left(\mathcal{C}^1\left(\left[0,T_f\right];\mathcal{C}^0\left(\overline{\Omega}\right)\right)\right)^2$.
\end{proof}
%

\begin{lem}\label{R2welldefi}
 The map $R_2:\left(\mathcal{C}^1\left(\left[0,T_f\right];\mathcal{C}^0\left(\overline{\Omega}\right)\right)\right)^2\rightarrow\mathcal{C}^0\left(\left[0,T_f\right];\mathcal{C}^0\left(\overline{\Omega}\right)\right)$ is well defined.
\end{lem}
\begin{proof}
Observe that the pair of constant functions $\left(\underline{T},\overline{T}\right)=\left(0,K\right)$ is a sub-super solution of $\left(\ref{EDPlin}\right)$ and and the reaction term in $\left(\ref{EDPlin}\right)$ is bounded a.e. $(t,x)\in (0,T_f)\times\Omega$ and for $T\in \left[\underline{T},\overline{T}\right]$ . Then, applying Theorem of  \cite[p. 94]{Deuel_1978}, there exists at least a weak solution $T$ of $\left(\ref{EDPlin}\right)$ such that $0\leq T\leq K$ a.e. in $\left(0,T_f\right)\times\Omega$.
\\
%

Since $T\in\left[0,K\right]$, we get that the application $\left(t,x,T_+\right)\to f_1(T_+(t,x), N_+(t,x), \Phi_+(t,x))$ is bounded in $L^{\infty}\left(0,T_f;L^{\infty}\left(\Omega\right)\right)$. Hence, applying Lemma \ref{w2p} since $T_0\in W_{n}^{2-2/p,p}\left(\Omega\right)$, we deduce that
$T\in V_p$.
\\

In particular, since $W_{n}^{2-2/p,p}\left(\Omega\right)\hookrightarrow\mathcal{C}^0\left(\overline{\Omega}\right)$, we get
$$T\in\mathcal{C}^0\left(\left[0,T_f\right];\mathcal{C}^0\left(\overline{\Omega}\right)\right).$$

The uniqueness of $T=R_2\left(N,\Phi\right)$ can be deduced by a comparison argument using the regularity of $\left(T,N,\Phi\right)$.

\end{proof}

Before proving that $R_2$ is continuous, we show the following result:

\begin{lem}\label{compact}
	For any bounded set $A$ of $\left(\mathcal{C}^1\left(\left[0,T_f\right];\mathcal{C}^0\left(\overline{\Omega}\right)\right)\right)^2$, then $R_2\left(A\right)$ is bounded in $V_p$ for some $p>3$, where $V_p$ is the Banach space defined in Lemma \ref{w2p}.
\end{lem}
	Notice that, by Aubion-Lions lemma (see \cite[Théoréme 5.1, p. 58]{Lions_1969}) and \cite[Corollary 4]{Simon_1987}, one has the compact embedding
	$$V_p\hookrightarrow\mathcal{C}^0\left(\left[0,T_f\right];\mathcal{C}^0\left(\overline{\Omega}\right)\right).$$

\begin{proof}
 Given $\left(N,\Phi\right)\in A$ a bounded set of $\left(\mathcal{C}^1\left(\left[0,T_f\right];\mathcal{C}^0\left(\overline{\Omega}\right)\right)\right)^2$, then
 
 \begin{equation}\label{N_Fi_acotado}
 \| N\|_{\mathcal{C}^1\left(\left[0,T_f\right];\mathcal{C}^0\left(\overline{\Omega}\right)\right)},\;\|\Phi\|_{\mathcal{C}^1\left(\left[0,T_f\right];\mathcal{C}^0\left(\overline{\Omega}\right)\right)}\leq\widetilde{\mathcal{C}}
 \end{equation}
 and there exists a unique $T=R_2\left(N,\Phi\right)$ solution of $\left(\ref{EDPlin}\right)$. Moreover, we have that the application $(t,x)\to f_1(T_+(t,x),N(t,x),\Phi(t,x))$ is bounded in $L^\infty\left(\left(0,T_f\right);L^\infty\left(\Omega\right)\right)$. Thus by Lemma \ref{w2p} and $\left(\ref{N_Fi_acotado}\right)$,
%
%
%
%

$$\Big\|T\Big\|_{V_p}\leq\mathcal{C}\left( \Big\|f_1\left(T_+,N_+,\Phi_+\right)\Big\|_{L^p\left(\left(0,T_f\right);L^p\left(\Omega\right)\right)},\;\Big\| T_0\Big\|_{W^{2-2/p,p}\left(\Omega\right)}\right)\leq\widehat{\mathcal{C}}.$$
\end{proof}

\begin{lem}\label{R2conti}
	 The map $R_2:\left(\mathcal{C}^1\left(\left[0,T_f\right];\mathcal{C}^0\left(\overline{\Omega}\right)\right)\right)^2\rightarrow\mathcal{C}^0\left(\left[0,T_f\right];\mathcal{C}^0\left(\overline{\Omega}\right)\right)$ is continuous.
\end{lem} 
\begin{proof}

Given
\begin{equation}\label{N_Fi_cont}
\left(N_n,\Phi_n\right)\rightarrow\left(N,\Phi\right)\;\;\text{in}\;\; \left(\mathcal{C}^1\left(\left[0,T_f\right];\mathcal{C}^0\left(\overline{\Omega}\right)\right)\right)^2
\end{equation}
we are going to check that $T_n=R_2\left(N_n,\Phi_n\right)\rightarrow T=R_2\left(N,\Phi\right)$ in $\mathcal{C}^0\left(\left[0,T_f\right];\mathcal{C}^0\left(\overline{\Omega}\right)\right)$. 
\\


Applying Lemma \ref{compact}, it holds that $T_n=R_2\left(N_n,\Phi_n\right)$ is bounded in $V_p$, hence there exists a subsequence $T_{n_k}\in V_p$ and a limit $T^*\in V_p$ such that
$$T_{n_k} \rightharpoonup T^*\;\text{weakly in}\;  V_p\;\text{and strongly in}\;\mathcal{C}^0\left(\left[0,T_f\right];\mathcal{C}^0\left(\overline{\Omega}\right)\right)$$
and 

$$\dfrac{\partial T_{n_k}}{\partial t}\rightharpoonup\dfrac{\partial T^*}{\partial t}\quad\text{weakly in }L^p\left(0,T_f;L^p\left(\Omega\right)\right).$$

In particular,

$$\Delta T_{n_k} \rightharpoonup\Delta T^*\quad\text{weakly in } L^p\left(0,T_f;L^p\left(\Omega\right)\right).$$

Using these convergences and $\left(\ref{N_Fi_cont}\right)$, the continuity of $f_1\left(T_+,N_+,\Phi_+\right)$ and the locally lipschitz property of the application $(T,N,\Phi)\to f_1(T_+,N_+,\Phi_+)$ respect to all the variables, we deduce 

$$f_1\left(\left(T_{n_k}\right)_+,\left(N_{n_k}\right)_+,\left(\Phi_{n_k}\right)_+\right)\rightarrow f_1\left(\left(T^*\right)_+,N_+,\Phi_+\right) \text{ strongly in } \mathcal{C}^0\left(\left[0,T_f\right];\mathcal{C}^0\left(\overline{\Omega}\right)\right).$$

Taking $n_k\rightarrow\infty$ we have that $T^*=R_2\left(N,\Phi\right)$ and since the solution of $\left(\ref{EDPlin}\right)$ is unique, then $T^*=T$ and

$$T_n\rightarrow T\;\text{in}\; \mathcal{C}^0\left(\left[0,T_f\right];\mathcal{C}^0\left(\overline{\Omega}\right)\right).$$
\end{proof}

From Lemmas $\ref{R1welldefi_conti}$, $\ref{R2welldefi}$ and $\ref{R2conti}$, we obtain that:
\begin{col}\label{Rcontinuo}
	The map
$$\mathbf{R}:\mathcal{C}^0\left(\left[0,T_f\right];\mathcal{C}^0\left(\overline{\Omega}\right)\right)\rightarrow\mathcal{C}^0\left(\left[0,T_f\right];\mathcal{C}^0\left(\overline{\Omega}\right)\right)$$ is well defined and continuous.
\end{col}

\paragraph{Step $\boldsymbol{3}$.} 
\begin{lem}\label{Rcompac}
	The operator $\mathbf{R}:\mathcal{C}^0\left(\left[0,T_f\right];\mathcal{C}^0\left(\overline{\Omega}\right)\right)\rightarrow\mathcal{C}^0\left(\left[0,T_f\right];\mathcal{C}^0\left(\overline{\Omega}\right)\right)$ is compact.
\end{lem}

\begin{proof}
	
Let $\widetilde{T}\in \mathcal{C}^0\left(\left[0,T_f\right];\mathcal{C}^0\left(\overline{\Omega}\right)\right)$ then by Lemmas $\ref{R1welldefi_conti}$ and $\ref{R2welldefi}$ there exists a unique $T=\mathbf{R}\left(\widetilde{T}\right)$ such that $0\leq T\leq K$ a.e. $\left(t,x\right)\in\left(0,T_f\right)\times\Omega$ and being $T$ the solution of $\left(\ref{EDPlin}\right)$.
\\

Moreover, there exists an unique $\left(N,\Phi\right)\in\left(\mathcal{C}^0\left(\left[0,T_f\right];\mathcal{C}^0\left(\overline{\Omega}\right)\right)\right)^2$ such that $0\leq N,\Phi\leq K$ for all $\left(t,x\right)\in\left[0,T_f\right]\times\overline{\Omega}$.
Hence, $f_1\left(T_+,N_+,\Phi_+\right)$ is bounded in $L^\infty\left(0,T_f;L^\infty\left(\Omega\right)\right)$, in particular, in $L^p\left(0,T_f;L^p\left(\Omega\right)\right)$ for all $p<\infty$. Following a similar argument of Lemma $\ref{compact}$, we obtain that $T$ is bounded in $V_p$ for some $ p>3$.
\\

Finally, applying the compact embedding of $V_p$ in $\mathcal{C}^0([0,T_f];\mathcal{C}^0(\overline\Omega))$, we obtain that $\mathbf{R}$ is compact from $\mathcal{C}^0\left(\left[0, T_f\right],\mathcal{C}^0\left(\overline{\Omega}\right)\right)$ to itself.

\end{proof}

\paragraph{Step $\boldsymbol{4}$.}
\begin{lem}\label{Rlamda}
	For any $T=\lambda\;\mathbf{R}\left(T\right)$, for some $\lambda\in\left[0,1\right]$, then $\Big\|T\Big\|_{\mathcal{C}^0\left(\left[0, T_f\right],\mathcal{C}^0\left(\Omega\right)\right)}\leq\mathcal{C}$ with $\mathcal{C}>0$ independent of $\lambda\in\left[0,1\right]$.
\end{lem}

\begin{proof}

For $\lambda=0$ the result is trivial, hence we suppose $\lambda\in\left(0,1\right]$. 
\\

On the one hand, if we  rewrite $f_1(T_+,N_+,\Phi_+)=T_+\; \widetilde {f_1}(T_+,N_+,\Phi_+)$, we have that,

$$\displaystyle T_t-\Delta\;T=\lambda\;f_1\left(T_+/\lambda,N_+,\Phi_+\right)=\lambda\;\dfrac{T_+}{\lambda}\;\widetilde{f_1}\left(T_+/\lambda,N_+,\Phi_+\right))\leq\rho\;T_+ \left(1-\dfrac{T_+}{\lambda\;K}\right).$$


Since $0\le T(0,x)\le K$ in $\overline\Omega$, we can argue similarly to Lemma $\ref{estimaciones}$ and conclude that $0\leq T\leq  K$ in $\left[0,T_f\right]\times\overline{\Omega}$. Thus, $T$ is bounded $\mathcal{C}^0\left(\left[0, T_f\right],\mathcal{C}^0\left(\overline{\Omega}\right)\right)$ independently of $\lambda\in\left[0,1\right]$.
\end{proof}

Finally, from Corollary $\ref{Rcontinuo}$, and Lemmas $\ref{Rcompac}$ and $\ref{Rlamda}$, the operator $\mathbf{R}$ satisfies the hypotheses of Theorem \ref{Leray}. Thus, we conclude that the map $\mathbf{R}$ has a fixed point $T=\mathbf{R}\left(T\right)$ which is a classical solution of $\left(\ref{problin}\right)$ and consequently it is also a classical solution of $\left(\ref{probOriginal}\right)$-$\left(\ref{condinicio}\right)$.

\end{proof}


%
%

%
%
%
%

\section{Asymptotic behaviour}\label{comp_asin}

\subsection{Stability of the (non-diffusion) ODE system}

Once we have proved the existence and uniqueness of solution for $\left(\ref{probOriginal}\right)$ for any finite time, let us study the long time behaviour of this solution. For that, first of all, we will study the non-diffusion problem

\begin{equation}\label{SDO}
\left\{\begin{array}{rcl}
\dfrac{d\; T}{d\; t} & = & f_1\left(T,N,\Phi\right) \\
&&\\
\dfrac{d\; N}{d\; t} & = & f_2\left(T,N,\Phi\right)\\
&&\\
\dfrac{d\; \Phi}{d\; t} & = &f_3\left(T,N,\Phi\right)
\end{array}\right.
\end{equation}
with initial data
\begin{equation}\label{inicioSDO}
\left(T,N,\Phi\right)\left(0\right)=\left(T_0,N_0,\Phi_0\right)\in\mathbb{R}^3
\end{equation}
%
%
%
%
%
%
%
such that $0\le T_0,\; N_0,\;\Phi_0\le K$ and the functions $f_i=f_i\left(T,N,\Phi\right)\in\mathbb{R}$ for $i=1,2,3$ are defined in $\left(\ref{funciones}\right)$. Since problem $\left(\ref{SDO}\right)$ is decoupled for each $x\in\Omega$, it suffices to study the ODE system $\left(\ref{SDO}\right)$ with a fixed  $\left(T_0,N_0,\Phi_0\right)\in\mathbb{R}^3_{+}$. First of all, we can deduce the same bounds for the solution $\left(T,N,\Phi\right)$ as in problem $\left(\ref{probOriginal}\right)$-$\left(\ref{condinicio}\right)$ and hence $\left(T\left(t\right),N\left(t\right),\Phi\left(t\right)\right)\in\mathbb{R}^3_{+}$ $\forall t\geq0$.
\\

In order to obtain the equilibrium points, we solve the nonlinear algebraic system $f_i\left(T,N,\Phi\right)=0$ for $i=1,2,3$. From $f_2\left(T,N,\Phi\right)=0$ we obtain that

$$\left\{\begin{array}{lcl}
T\;\sqrt{1-\left(P\left(\Phi,T\right)\right)^2}=0,\\
\\
T\;\Phi=0,\\
\\
N\left(T+\Phi\right)=0.
\end{array}
\right.$$	

From $T\;\sqrt{1-\left(P\left(\Phi,T\right)\right)^2}=0$ and $T\;\Phi=0$, we have $T=0$. From the third condition, we obtain that $N=0$ or $\Phi=0$.
\\

Thus, the equilibria of $\left(\ref{SDO}\right)$ are
\vspace{-0.1cm}
\begin{equation}\label{solequilibrio}
\begin{split}
\text{\textbullet}\;\; P_1&=\left\{\left(0,0,0\right)\right\}.\\
 \text{\textbullet}\;\; P_2&=\left\{\left(0,N,0\right),\quad N>0\right\}.\\
 \text{\textbullet}\;\; P_3&=\left\{\left(0,0,\Phi\right),\quad\Phi>0\right\}.
\end{split}
\end{equation}

\begin{obs}
	Observe that $P_1\cup P_2\cup P_3$ is a continuum of equilibria points.
\end{obs}

\begin{obs}
The linearisation technique around the equilibria $P_1$, $P_2$ and $P_3$ doesn't give any relevant information because one of the eigenvalue of this linearisation is zero.

\end{obs}

Now, we consider the differential equation for the sum $S=T+N+\Phi$, which satisfies

\begin{equation}\label{eqSuma}
\left\{
\begin{array}{l}
\dfrac{d\;S}{d t}=\underbrace{\left(\rho\;T\;P\left(\Phi,T\right)+\dfrac{\gamma\;\Phi}{K}\;T\;\sqrt{1-\left(P\left(\Phi,T\right)\right)^2}\right)}_{=0\text{  if  } T=0\text{  or  } \Phi=0}\left(1-\dfrac{S}{K}\right),\\
\\
S\left(0\right)=S_0:=T_0+N_0+\Phi_0.
\end{array}
\right.
\end{equation}

Hence, we see that $S\left(t\right)$ is increasing if $S_0< K$, and, $S\left(t\right)\nearrow S_*\leq K$ as $t\rightarrow+\infty$. On the other hand, if $S_0=K$, then $S\left(t\right)=K\quad\forall t\in\left[0,+\infty\right)$. Finally, if $S_0>K$ then $S\left(t\right)$ is decreasing and $S\left(t\right)\searrow S_*\leq K$ as $t\rightarrow+\infty$. For brevity, we only study the case $S_0\le K$. 
\\

We show two particular cases:
\begin{itemize}
	\item If we consider $T_0=0$ then $T(t)=0$ $\forall$ $t\ge 0$. It implies from $\left(\ref{eqSuma}\right)$ that $\dfrac{d\;S\left(t\right)}{dt}=0$ $\forall t>0$. Hence $S\left(t\right)=N_0+\Phi_0$ $\forall t>0$. In terms of the subsystem $\left(N,\Phi\right)$, we obtain that 
	$$\left\{\begin{array}{l}
		\dfrac{d\;N}{d t}=\beta_2\;N\;\Phi,\\
		\\
		\dfrac{d\;\Phi}{d t}=-\beta_2\;N\;\Phi,
	\end{array}	\right.$$
	with $N\left(0\right)=N_0\ge0$ and $\Phi\left(0\right)=\Phi_0\geq0$. Hence $N\left(t\right)\nearrow N_*$ and, $\Phi\left(t\right)\searrow0$ with $N_*=N_0+\Phi_0$. 
	\item  If we consider $\Phi_0=0$ then $\Phi\left(t\right)=0$ $\forall t>0$.  Hence, from $\left(\ref{eqSuma}\right)$, $S\left(t\right)=T_0+N_0$ for all $t>0$. Since $N\left(t\right)\nearrow N_*$ then, $T\left(t\right)\searrow 0$ with $N_*=N_0+T_0$.
\end{itemize}

In order to study the stability of $\left(\ref{SDO}\right)$, we use the properties of omega limit sets. Given $y_0=\left(T_0,N_0,\Phi_0\right)\in\mathbb{R}_{+}^{3}$ such that $T_0+N_0+\Phi_0\le K$, then there exists a unique solution  $y\left(t\right)=\left(T,N,\Phi\right)\left(t\right)\in\mathbb{R}^3_{+}$ $\forall t\in\left[0,+\infty\right)$ of $\left(\ref{SDO}\right)$-$\left(\ref{inicioSDO}\right)$ such that $T(t)+N(t)+\Phi(t)\le K$. Therefore the corresponding $\omega$-limit set is defined by

$$\omega\left(y_0\right)=\{y_*\in\mathbb{R}^3_{+},\;\exists\;t_n\rightarrow\infty\;:y\left(t_n\right)\rightarrow y_*\text{  in  }\mathbb{R}^3\}$$
and is a nonempty compact and invariant set of $\mathbb{R} ^3_+$. Since $0<S_0\leq K$ then $S\left(t\right)\nearrow S_*\leq K$ and $N\left(t\right)\nearrow N_*\leq S_*$ for $t\rightarrow\infty$, where $S_*=S_*\left(S_0\right)$ and $N_*=N_*\left(T_0,N_0,\Phi_0\right)$, because both functions are increasing and bounded from above. Therefore,

\begin{equation}\label{omegalimite}
\small\omega\left(T_0,N_0,\Phi_0\right)\subseteq\left\{\left(\widetilde{T},N_*,S_*-N_*-\widetilde{T}\right),\;\widetilde{T}\in\big[0,S_*-N_*\big]\right\}.
\end{equation}


\begin{teo}
	Given $y_0=\left(T_0,N_0,\Phi_0\right)\in\mathbb{R}_{+}^{3}$ and $S_0=T_0+N_0+\Phi_0\leq K$. If $y_0\neq\left(0,0,\Phi_0\right)$ with $\Phi_0\geq0$, then the $\omega$-limit set is a unitary set
	
	$$\omega\left(T_0,N_0,\Phi_0\right)=\Big\{\left(0,N_*,0\right)\Big\}.$$
	
\end{teo}
\begin{obs}
	If $y_0\notin P_1\cup P_3$, then $\omega\left(y_0\right)$ is unitary and belongs to $P_2$.
\end{obs}
\begin{proof}
	
Let $\left(T_p,N_p,\Phi_p\right)$ be the solution starting from a point $$p=\left(\widetilde{T},N_*,S_*-N_*-\widetilde{T}\right)\in\omega\left(T_0,N_0,\Phi_0\right).$$

Since $\omega\left(T_0,N_0,\Phi_0\right)$ is an invariant set, it holds that $N_p\left(t\right)=N_*$ $\forall t$. Hence $\dfrac{d}{dt}N_p=0$. Now, from the $N_p$ equation,

\begin{equation}\label{eq2}
0=\dfrac{d}{dt}N_p=T_p\left(\alpha\;\sqrt{1-\left(\dfrac{\Phi_p}{\Phi_p+T_p}\right)^2}+\beta_1\; N_p\right) +\Phi_p\left(\delta\;T_p +\beta_2\;N_p\right)\geq \beta_1\;T_p\;N_p.
\end{equation}

Hence $T_p\left(t\right)=0$ for all $t\geq0$ and $\widetilde{T}=0$. Then, $p=\left(0,N_*,S_*-N_*\right)$. 
\\

Since in particular $p$ is a equilibrium point and $N_*>0$, then $p$ must be a point of type $P_2$, hence $p=\left(0,N_*,0\right)$. In particular, $N_*=S_*$.
\end{proof}

As consequence of this result, we deduce:

\begin{col}
	$P_3$ is a continuum of unstable equilibria. Indeed, for any $\left(T_0,N_0,\Phi_0\right)$ with $T_0>0$ or $N_0>0$, then its solution satisfies
	$$\left(T\left(t\right),N\left(t\right),\Phi\left(t\right)\right)\rightarrow\left(0,N_*,0\right)\;\;\text{as}\;\; t\to\infty$$
	with $N_*\geq T_0+N_0+\Phi_0$.
\end{col}

\subsection{Stability of the  Diffusion Model $\boldsymbol{\left(\ref{probOriginal}\right)$-$\left(\ref{condinicio}\right)}$}

In this section, we study the stability of the constant equilibria of $\left(\ref{probOriginal}\right)$-$\left(\ref{condinicio}\right)$ to spatio-temporal perturbations for $t\rightarrow+\infty$. The constant solutions of $\left(\ref{probOriginal}\right)$-$\left(\ref{condinicio}\right)$ are the same of $\left(\ref{SDO}\right)$-$\left(\ref{inicioSDO}\right)$ given in $\left(\ref{solequilibrio}\right)$. In this case, the main difference is that we do not have a differential problem for $S\left(t\right)$ as in $\left(\ref{eqSuma}\right)$. Before showing the results we will discuss two remarks.

\begin{obs}
	The condition $N_0(x)>0$ for $x\in \overline{\Omega}$ used in the followings results can be relaxed by $N(t_*,x)>0$ for some $t_*>0$ since $N\left(\cdot,x\right)$ is increasing in time. On the other hand, applying the strong maximum principle to the parabolic problem that satisfies the tumor variable $T$ (since the reaction can be rewrite as $f_1(T,N,\Phi)=T\; \widetilde{f}_1(T,N,\Phi)$ with $\widetilde{f}_1$ bounded), it holds that $T(t_1,x)>0$ for any $t_1>0$ and for all $x\in \overline{\Omega}$. Due to the hypoxia term, in particular one has $\dfrac{\partial\;N}{\partial\;t}\geq \alpha\;T\;\sqrt{1-P^2\left(\Phi,T\right)}>0$ hence $N\left(t_*,x\right)>0$ for any $t_*>t_1$ and for all $x\in \overline{\Omega}$.
\end{obs}
We introduce some results of pointwise and uniform convergence as time goes to infinity. First of all, we will see that vasculature always goes to zero.

\begin{lem}\label{Fi_0}
	Given a solution $\left(T,N,\Phi\right)$ of $\left(\ref{probOriginal}\right)$-$\left(\ref{condinicio}\right)$, then for each $x\in\overline{\Omega}$ such that $N_0\left(x\right)>0$ one has $\Phi\left(t,x\right)\rightarrow0$ when $t\rightarrow+\infty$.
	
\end{lem}

\begin{proof}
Let $x\in\overline{\Omega}$ such that $N_0\left(x\right)>0$. Since $N\left(\cdot,x\right)$ is increasing, then, $0<N_0\left(x\right)\leq N\left(t,x\right)$ for all $t>0$. Now, we separate this proof in two cases depending on the value of $N\left(t,x\right)$:
\begin{enumerate}[a)]
	\item If there exists $t_*>0$ such that $N\left(t,x\right)\geq K$ for all $t\geq t_*$, then we get $f_3\left(T\left(t,x\right),N\left(t,x\right),\Phi\left(t,x\right)\right)\leq-\beta_2\;K\;\Phi\left(t,x\right)$ for all $t\geq t_*$. Hence we have the following 
		\begin{equation}\label{AcotacionF1}
	\left\{\begin{array}{l}
	\dfrac{\partial\Phi}{\partial t}\left(t,x\right)\leq -\beta_2\;K\;\Phi\left(t,x\right)\;\;\text{   in   }\;\;\left[t_*,+\infty\right),\\
	\\
	\Phi\left(t_*,x\right)\ge 0.
	\end{array}\right.
	\end{equation}
	
	Therefore for all $t\geq t_*$,
	$$\Phi\left(t,x\right)\leq\Phi\left(t_*,x\right)\;e^{-\beta_2\;K\;\left(t-t_*\right)}\rightarrow0\qquad\text{as}\;\; t\rightarrow+\infty.$$
	
	\item If $N\left(t,x\right)<K$ for all $t\geq0$ we reason by contradiction. Assume that there exists a sequence $\left\{t_n\right\}_{n\in\mathbb{N}}$ such that $t_n\rightarrow+\infty$ and $t_{n+1}-t_n\geq1$ for all $n\in\mathbb{N}$ and there exists $\eta\left(x\right)$ such that $\Phi\left(t_n,x\right)\geq\eta\left(x\right)>0$ for all $n\in\mathbb{N}$. Since 
	
	\begin{equation}\label{AcotacionN}
	\left\{\begin{array}{l}
	\dfrac{\partial N}{\partial t}\left(t,x\right)\geq \beta_2\;N\left(t,x\right)\;\Phi\left(t,x\right)\;\;\text{   in   }\;\;\left(0,+\infty\right),\\
	\\
	N\left(0,x\right)\ge 0,
	\end{array}\right.
	\end{equation}
	we have the following estimates for $N\left(t,x\right)$.
	
	\begin{equation}\label{sol_N}
	\displaystyle K> N\left(t,x\right)\geq N_0\left(x\right)\;e^{\beta_2\;\int_{0}^{t}\Phi\left(s,x\right)\;ds}.
	\end{equation}
	Since $0\leq T,\;\Phi\leq K$ and $N\left(t,x\right)<K$ for all $t>0$, we get the following lower bound
	
	$$\dfrac{\partial \Phi}{\partial\;t}=f_3\left(T,N,\Phi\right)\geq-\dfrac{\gamma}{K}\;T\;\sqrt{1-P\left(\Phi,T\right)^2}\;\Phi\left(\dfrac{T+\Phi}{K}\right)-\delta\;T\;\Phi-\beta_2\;N\;\Phi\geq$$
	
	$$\geq -2\;\gamma\;\Phi-\delta\;K\;\Phi-\beta_2\;K\;\Phi=-\mathcal{C}_0\;\Phi.$$
	
	Hence,
	$$\Phi\left(t,x\right)\geq e^{-\mathcal{C}_0\left(t-t_n\right)}\Phi\left(t_n,x\right)\geq e^{-\mathcal{C}_0\left(t-t_n\right)}\eta\left(x\right)\;\;\forall t\in\left(t_n,t_{n+1}\right).$$
	Integrating in $\left[t_n,t_{n+1}\right]$ and using that $t_{n+1}-t_n\geq1$
	$$\int_{t_n}^{t_{n+1}}\Phi\left(t,x\right)\;dt\geq\dfrac{\eta\left(x\right)}{\mathcal{C}_0}\left(1-e^{-\mathcal{C}_0\left(t_{n+1}-t_n\right)}\right)\ge \dfrac{\eta(x)}{C_0}(1-e^{-C_0})>0.$$
	
	 Finally, adding all $t_n$
	$$\sum_{n=1}^{+\infty}\int_{t_n}^{t_{n+1}}\Phi\left(t,x\right)\;dt=+\infty.$$
	
	Hence, $\displaystyle\int_{0}^{+\infty}\Phi\left(t,x\right)\;dt=+\infty$ and we arrive at contradiction with $\left(\ref{sol_N}\right)$ and the proof is completed.
	
\end{enumerate}
\end{proof}

As consequence of Lemma $\ref{Fi_0}$, we  deduce:
\begin{col}\label{P3_inestable}
	The equilibria $P_3$ are unstable.
\end{col}


\begin{obs}
	As consequence of $t\mapsto N(t,x)$ is increasing for all $x\in\overline{\Omega}$, we deduce that $P_1$ is not asymptotically stable.
\end{obs}

In the following result, adding a constraint on some parameters of the problem, we can deduce the behaviour of the solution of the system $\left(\ref{probOriginal}\right)$-$\left(\ref{condinicio}\right)$ as $t\rightarrow+\infty$.

\begin{lem}\label{lemaestabilidad1}
	Given a classical solution $\left(T,N,\Phi\right)$ of $\left(\ref{probOriginal}\right)$-$\left(\ref{condinicio}\right)$ such that $N_0\left(x\right)>0$ for all $x\in\overline{\Omega}$ and assume that
		
	\begin{equation}\label{CondiNlejosKdelta}
	\delta\ge\dfrac{\gamma}{K}.
	\end{equation}
	Then, for all $t\geq0$:
	
	$$\|\Phi\left(t,\cdot\right)\|_{\mathcal{C}^0\left(\overline{\Omega}\right)}\leq \|\Phi_0\|_{\mathcal{C}^0\left(\overline{\Omega}\right)}\;e^{-\beta_2\;N_{0}^{\min}\;t},$$ 
	where $\displaystyle N_{0}^{\min}=\min_{x\in\overline{\Omega}} N_0\left(x\right)$. In addition, there exists $\mu\in\left(0,\min\left\{\beta_1,\;\beta_2\right\}N_{0}^{\min}\right)$ such that

	$$\| T\left(t,\cdot\right)\|_{\mathcal{C}^0\left(\overline{\Omega}\right)}\leq M\;e^{-\mu\;t},\;\;\forall t\geq0$$	
	with $M=\max\left\{\| T_0\|_{\mathcal{C}^0\left(\overline{\Omega}\right)},\;\dfrac{\rho\;\| \Phi_0\|_{\mathcal{C}^0\left(\overline{\Omega}\right)}}{\beta_1\;N_{0}^{\min}-\mu}\right\}>0$. Moreover, there exists $N_{\max}>0$ such that
  
    $$N\left(t,x\right)\leq N_{\max}\;\;\forall \left(t,x\right)\in\left[0,+\infty\right)\times\overline{\Omega}.$$
    
%
\end{lem}

\begin{proof}
    Using hypothesis $\left(\ref{CondiNlejosKdelta}\right)$ and the bounds $T,\;\Phi\geq0$ and $N\geq N_0$, we can estimate
	
	$$f_3\left(T,N,\Phi\right)\leq\dfrac{\gamma}{K}\;\Phi\;T-\delta\;\Phi\;T-\beta_2\;\Phi\;N\leq-\beta_2\;\Phi\;N_0.$$

	Hence, $\Phi$ satisfies the differential inequality problem
	
		\begin{equation}\label{AcotacionF}
	\left\{\begin{array}{l}
	\dfrac{\partial\Phi}{\partial t}\leq -\beta_2\;\Phi\;N_{0}\left(x\right)\;\;\text{   in   }\;\;\left[0,+\infty\right)\times\overline{\Omega,}\\
	\\
	\Phi\left(0,x\right)=\Phi_0\left(x\right)\leq \| \Phi_0\|_{\mathcal{C}^0\left(\overline{\Omega}\right)}\;\;\text{  in  }\;\;\overline{\Omega},
	\end{array}\right.
	\end{equation}

	Using that $N_{0}^{\min}>0$, we conclude that
	
	\begin{equation}\label{Fi_acotada}
	\Phi\left(t,x\right)\leq\Phi_0\left(x\right)\;e^{-\beta_2\;N_{0}\left(x\right)\;t}\leq \| \Phi_0\|_{\mathcal{C}^0\left(\overline{\Omega}\right)}\;e^{-\beta_2\;N_{0}^{\min}\;t}.
	\end{equation}
	
In particular, $\Phi(t,x)\to 0$ as $t\rightarrow+\infty$ uniformly in $x\in\overline{\Omega}$.
   \\

    Using $\left(\ref{Fi_acotada}\right)$ and the bounds $T,\;\Phi\geq0$, $N\geq N_0$ and $P\left(\Phi,T\right)\; T\le \Phi$, we can estimate $f_1\left(T,N,\Phi\right)$ as follows: 
	$$f_1\left(T,N,\Phi\right)\leq \rho\;P\left(\Phi,T\right)\;T-\alpha\;T\sqrt{1-P^2\left(\Phi,T\right)}-\beta_1\;N\;T\leq$$
	$$\leq\rho\;\Phi-\beta_1\;N\;T\leq\rho\;\| \Phi_0\|_{L^\infty\left(\Omega\right)}\;e^{-\beta_2\;N_{0}^{\min}\;t}-\beta_1\;N_{0}^{\min}\;T.$$

	Therefore, $T\leq S$, where $S$ is the unique solution of the following parabolic problem

	\begin{equation}\label{AcotacionT}
	\left\{\begin{array}{l}
	\dfrac{\partial S}{\partial t}-\Delta\;S= \rho\;\| \Phi_0\|_{\mathcal{C}^0\left(\overline{\Omega}\right)}\;e^{-\beta_2\;N_{0}^{\min}\;t}-\beta_1\;N_{0}^{\min}S\quad\text{  in  }\left(0,+\infty\right)\times\overline{\Omega},\\
	\\
	S\left(0,x\right)= \| T_0\|_{\mathcal{C}^0\left(\overline{\Omega}\right)}\;\;\text{  in  }\Omega,\\
	\\
	\dfrac{\partial S}{\partial n}\Bigg\vert_{\partial\Omega}=0.
	\end{array}
	\right.
	\end{equation}

	Now, we can find a super solution of $\left(\ref{AcotacionT}\right)$ with the form $$\overline{T}\left(t\right)=M\;e^{-\mu\;t}$$
	such that $\min\left\{\beta_1,\;\beta_2\right\}N_{0}^{\min}> \mu>0$ and $M=\max\left\{\| T_0\|_{\mathcal{C}^0\left(\overline{\Omega}\right)},\;\dfrac{\rho\;\| \Phi_0\|_{\mathcal{C}^0\left(\overline{\Omega}\right)}}{\beta_1\;N_{0}^{\min}-\mu}\right\}>0$ for all $t\geq0$. Consequently, $\overline{T}$ is a super solution of $\left(\ref{AcotacionT}\right)$ and we have
	
	\begin{equation}\label{T_acotada}
	T\left(t,x\right)\leq S\left(t,x\right)\leq \overline{T}\left(t\right)= M\;e^{-\mu\;t}.
	\end{equation}
	
    In particular, $T(t,x)\to 0$ as $t\rightarrow+\infty$ uniformly for $x\in\overline{\Omega}$.
	\\
	
	Then, we can obtain a uniform upper bound in time and space for $N\left(t,x\right)$ since,
	
	\begin{equation}\label{eqN_t}
	\dfrac{\partial\;N}{\partial\; t}=a\left(t,x\right)\;N+b\left(t,x\right)
	\end{equation}
	where
	
	$$a\left(t,x\right)=\beta_2\;\Phi\left(t,x\right)+\beta_1\;T\left(t,x\right) $$
	and
	$$b\left(t,x\right)=\alpha\;B\left(\Phi\left(t,x\right),T\left(t,x\right)\right)+\delta\;\Phi\left(t,x\right)\;T\left(t,x\right).$$ 
	
	Hence, one has the variation of constants formula
	
		\begin{equation}\label{eqN}
	\displaystyle N\left(t,x\right)=\left(N_0\left(x\right)+\int_{0}^{t}b\left(s,x\right)\;e^{-\;A\left(s,x\right)}\;ds\right){e}^{A\left(t,x\right)}\end{equation}
	with $\displaystyle A\left(t,x\right)= \int_{0}^{t}a\left(s,x\right)\;ds $.
	\\
	
	Using now the exponential upper bounds of $\Phi\left(t,x\right)$ and $T\left(t,x\right)$ given in $\left(\ref{Fi_acotada}\right)$ and $\left(\ref{T_acotada}\right)$ respectively, 
	$$a\left(t,x\right)\leq \widehat{a}\left(t\right)=\beta_2\;\| \Phi_0\|_{\mathcal{C}^0\left(\overline{\Omega}\right)}e^{-\beta_2\;N_{0}^{\min}\;t}+\beta_1\;M\;e^{-\mu\;t},$$
	$$b\left(t,x\right)\leq\widehat{b}\left(t\right)=\alpha\;M\;e^{-\mu\;t}+\delta\| \Phi_0\|_{\mathcal{C}^0\left(\overline{\Omega}\right)}\;e^{-\beta_2\;N_{0}^{\min}\;t}\;M\;e^{-\mu\;t},$$
	and then,
	$$A\left(t,x\right)=\int_{0}^{t}a\left(s,x\right)\;ds\leq\int_{0}^{t}\widehat{a}\left(s\right)\;ds\le\mathcal{C}_1\;\;\text{and}\;\;\int_0^t b(s,x) e^{-A(s,x)}\;ds\le\mathcal{C}_2.$$

	Hence, we conclude that there exists a constant $N_{\max}>0$ such that $N\left(t,x\right)\leq N_{\max}$ $\forall\left(t,x\right)\in\left[0,+\infty\right)\times\overline{\Omega}$. Since $N\left(t,x\right)$ is increasing, it holds that there exists $N_*\left(x\right)\leq N_{\max}$ such that $N\left(t,x\right)\rightarrow N_*\left(x\right)\leq N_{\max}$ pointwise in space when $t\rightarrow+\infty$.
	
\end{proof}
	
Our third result shows that when $\beta_1$ is large with respect to $\rho$ (that is, destruction of tumor by necrosis dominates to tumor growth), then the tumor tends to the extinction. For that, we need to introduce some notation. Given $b\in L^\infty(\Omega)$ we denote by $\lambda_1(-\Delta+b)$ the first eigenvalue of the problem
$$
\left\{\begin{array}{ll}
-\Delta u+b(x)u=\lambda u & \mbox{in $\Omega$,}\\
\\
\dfrac{\partial u}{\partial n}=0 & \mbox{on $\partial\Omega$.} 
\end{array}
\right.
$$
\begin{lem}\label{lema_autovalor}
	Given a classical solution $\left(T,N,\Phi\right)$ of $\left(\ref{probOriginal}\right)$-$\left(\ref{condinicio}\right)$ such that $N_0\left(x\right)>0$ for all $x\in\overline{\Omega}$ and assume that
	\begin{equation}\label{condi}
	\rho< \lambda_1(-\Delta +\beta_1 N_0(x)). 
	\end{equation}
	
	Then, for all $t\ge0$:

	$$\| T\left(t,\cdot\right)\|_{\mathcal{C}^0\left(\overline{\Omega}\right)}\leq \| T_0\|_{\mathcal{C}^0\left(\overline{\Omega}\right)}\;e^{-\left(\lambda_1\left(-\Delta+\beta_1\;N_0\left(x\right)\right)-\rho\right)\;t}\quad\forall t>0$$
	and there exists $0<\mu_*<\beta_2\;N_0^{\min}$ and $t_*>0$ large enough, such that 
	$$\|\Phi\left(t,\cdot\right)\|_{\mathcal{C}^0\left(\overline{\Omega}\right)}\leq \|\Phi\left(t_*,\cdot\right)\|_{\mathcal{C}^0\left(\overline{\Omega}\right)}\;e^{-\mu_*\left(t-t_*\right)}\quad\forall t>t_*.$$ 
	Moreover, there exists $N_{\max}>0$ such that
	
	$$N\left(t,x\right)\leq N_{\max}\;\;\forall \left(t,x\right)\in\left[0,+\infty\right)\times\overline{\Omega}.$$
	
%
\end{lem}
\begin{proof}
	Since $N(t,x)\geq N_0(x)$ for all $x\in\overline{\Omega}$ and $t\geq 0$, and using the positivity of $\Phi$ and that $0\leq P(\Phi,T)\leq 1$, we get
	$$
	T_t-\Delta T\leq  \rho T\left(1-\frac{T}{K}\right)-\beta_1 \;N_0(x)\;T.
	$$
	
	Hence, 
	\begin{equation}
	\label{desi}
	T(x,t)\leq S(x,t)\quad \forall t\geq 0, \forall x\in \Omega,
	\end{equation}
	where $S$ is the unique positive solution of the classical logistic equation
	\begin{equation}
	\label{log}
	\left\{\begin{array}{ll}
	S_t-\Delta S+ \beta_1\; N_0(x)\;S=\rho \;S\left(1-\frac{S}{K}\right) & \mbox{$t>0$,\; $x\in \Omega$,}\\
	\\
	S(x,0)= T_0(x) & x\in \Omega,\\
	\\
	\dfrac{\partial S}{\partial n}=0  & \mbox{$t>0$,\; $x\in \partial\Omega$.}
	\end{array}\right.
	\end{equation}
	
	Now, it is known (see for instance \cite{Wiley_2003}) that if $\rho$ satisfies (\ref{condi}) then the problem $\left(\ref{log}\right)$ has a super solution with the form
	
	$$
	\overline{S}(t,x)= \| T_0\|_{\mathcal{C}^0\left(\overline{\Omega}\right)}\;e^{-\left(\lambda_1\left(-\Delta+\beta_1\;N_0\left(x\right)\right)-\rho\right)\;t}\varphi\left(x\right)
	$$
	where $\varphi\left(x\right)$ is the positive eigenfunction associated with $\lambda_1\left(-\Delta+\beta_1\;N_0\left(x\right)\right)$ with $\|\varphi\|_\infty=1$. Consequently, we have that
	
	\begin{equation}\label{T_acotada_lambda}
	T\left(t,x\right)\leq S\left(t,x\right)\leq \overline{S}\left(t,x\right)\leq\| T_0\|_{\mathcal{C}^0\left(\overline{\Omega}\right)}\;e^{-\left(\lambda_1\left(-\Delta+\beta_1\;N_0\left(x\right)\right)-\rho\right)\;t}\rightarrow 0
	\end{equation}
	as $t\rightarrow+\infty$ uniformly for $x\in\overline{\Omega}$.
	\\
	
	Now, since $T,N,\Phi\ge0$, we can bound $f_3\left(T,N,\Phi\right)$ as follows:
	
	$$f_3\left(T,N,\Phi\right)\leq \dfrac{\gamma}{K}\;T\;\Phi\left(1-\dfrac{T+N+\Phi}{K}\right)-\beta_2\;N\;\Phi\le\Phi\left(\dfrac{\gamma}{K}\;T-\beta_2\;N\right)\le \Phi\left(\dfrac{\gamma}{K}\;T-\beta_2\;N_0\left(x\right)\right).$$
	
	Since $N_0\left(x\right)>0$ for all $x\in\overline{\Omega}$ and using $\left(\ref{T_acotada_lambda}\right)$, there exist $t_*>0$ large enough such that for all $t\ge t_*$, $f_3\left(T,N,\Phi\right)\le -\mu_0\;\Phi$ with $0<\mu_*<\beta_2\;N_0^{\min}$. Hence, $\Phi$ satisfies the differential inequality problem
	
		\begin{equation}\label{AcotacionF_lambda}
	\left\{\begin{array}{l}
	\dfrac{\partial\Phi}{\partial t}\leq -\mu_*\;\Phi\;\;\text{   in   }\;\;\left[t_*,+\infty\right)\times\overline{\Omega,}\\
	\\
	\Phi\left(t_*,x\right)=\|\Phi\left(t_*,x\right)\|_{\mathcal{C}^0\left(\overline{\Omega}\right)}\;\;\text{  in  }\;\;\overline{\Omega},
	\end{array}\right.
	\end{equation}
	
	Solving $\left(\ref{AcotacionF_lambda}\right)$, we conclude that 
	
	\begin{equation}\label{F_acotada_lambda}
	\Phi\left(t,x\right)\le \|\Phi\left(t_*,\cdot\right)\|_{\mathcal{C}^0\left(\overline{\Omega}\right)}\;e^{-\mu_0\left(t-t_*\right)}.
	\end{equation}
	
	In particular, $\Phi\left(t,x\right)\to0$ as $t\rightarrow+\infty$ uniformly in $x\in\overline{\Omega}$.
	\\
	
	Using now the exponential upper bounds of $\Phi\left(t,x\right)$ and $T\left(t,x\right)$ given in $\left(\ref{F_acotada_lambda}\right)$ and $\left(\ref{T_acotada_lambda}\right)$ respectively, we can obtain an uniform upper bound in time and space for $N\left(t,x\right)$ using the same argument that in $\left(\ref{eqN_t}\right)$ with the following estimates
	$$a\left(t,x\right)\leq \widehat{a}\left(t\right)=\beta_2\;\|\Phi\left(t_*,\cdot\right)\|_{\mathcal{C}^0\left(\overline{\Omega}\right)}\;e^{-\mu_*\left(t-t_*\right)}+\beta_1\;\| T_0\|_{\mathcal{C}^0\left(\overline{\Omega}\right)}e^{-\left(\lambda_1\left(-\Delta+\beta_1\;N_0\left(x\right)\right)-\rho\right)\;t},$$
	$$\begin{array}{ll}
	b\left(t,x\right)\leq\widehat{b}\left(t\right)=&\alpha\;\| T_0\|_{\mathcal{C}^0\left(\overline{\Omega}\right)}\;e^{-\left(\lambda_1\left(-\Delta+\beta_1\;N_0\left(x\right)\right)-\rho\right)\;t}+\\
	\\
	&+\delta\|\Phi\left(t_*,\cdot\right)\|_{\mathcal{C}^0\left(\overline{\Omega}\right)}\;e^{-\mu_*\left(t-t_*\right)}\;\| T_0\|_{\mathcal{C}^0\left(\overline{\Omega}\right)}\;e^{-\left(\lambda_1\left(-\Delta+\beta_1\;N_0\left(x\right)\right)-\rho\right)\;t}
	\end{array}.$$
	and then,
	$$A\left(t,x\right)=\int_{t_*}^{t}a\left(s,x\right)\;ds\leq\int_{t_*}^{t}\widehat{a}\left(s\right)\;ds\le\mathcal{C}_1\;\;\text{and}\;\;\int_{t_*}^t b(s,x) e^{-A(s,x)}\le\mathcal{C}_2.$$

	With a similar reasoning to the used in Lemma \ref{lemaestabilidad1} we conclude the existence of $N_*\in L^\infty(\Omega)$ and $N_{\max}>0$ such that $N(t,x)\to N_*(x)\leq N_{\max}$ pointwise in space when $t\rightarrow+\infty$.
\end{proof}

\begin{obs}
	It is well-known that the map $\beta_1\mapsto \lambda_1(-\Delta+\beta_1 \;N_0(x))$ is continuous and increasing. Moreover, if $N_0(x)>0$ for $x\in \overline\Omega$ we have that $ \lambda_1(-\Delta+\beta_1 N_0(x))\to \infty$ as $\beta_1\to \infty$. Hence,  given $\rho>0$ there exists $\beta_0(\rho)>0$ large enough such that for $\beta_1\geq \beta_0(\rho)$, condition (\ref{condi}) holds, and then the tumor tends to zero.
\end{obs}

\begin{col}\label{corollary_estabilidad_1}
	Assume hypotheses of Lemma $\ref{lemaestabilidad1}$ or Lemma $\ref{lema_autovalor}$ and given $N_*\left(x\right)\in\mathcal{C}^0\left(\overline{\Omega}\right)$ such that $N_*\left(x\right)\geq N_{*}^{\min}>0$ $\forall x\in\overline{\Omega}$, then the semi-trivial steady solution $\left(0,N_*,0\right)$ is locally stable in $\mathcal{C}^0\left(\overline{\Omega}\right)$.
\end{col}
\begin{proof}
	Let $\epsilon>0$ and $\|T_0\|_{\mathcal{C}^0\left(\overline{\Omega}\right)},\;\|\Phi_0\|_{\mathcal{C}^0\left(\overline{\Omega}\right)},\;\|N_0-N_*\|_{\mathcal{C}^0\left(\overline{\Omega}\right)}\leq\overline{\delta}$ for $\overline{\delta}>0$ to choice in function of $\epsilon$. Following the same argument that in Lemma $\ref{lemaestabilidad1}$ or Lemma $\ref{lema_autovalor}$, it is possible to prove (details are left to the reader) that $\|T\left(t,\cdot\right)\|_{\mathcal{C}^0\left(\overline{\Omega}\right)},\;\|\Phi\left(t,\cdot\right)\|_{\mathcal{C}^0\left(\overline{\Omega}\right)},\;\|N\left(t,\cdot\right)-N_*\|_{\mathcal{C}^0\left(\overline{\Omega}\right)}\leq\epsilon$ for all $t\ge 0$.
\end{proof}

	In the following result, we are able to know the long time behaviour of the system $\left(\ref{probOriginal}\right)$-$\left(\ref{condinicio}\right)$ when $N_0\left(x\right)$ is close to the capacity $K$.
	
  \begin{lem}\label{lemaestabilidad2}
  	Let $\epsilon>0$ small enough such that $N_0(x)\ge K-\epsilon$ for all $x\in\overline{\Omega}$. Then, the classical solution $\left(T,N,\Phi\right)$ of $\left(\ref{probOriginal}\right)$-$\left(\ref{condinicio}\right)$ satisfies
  	
  	$$\begin{array}{l}
  	T\left(t,x\right)\leq \| T_0\|_{\mathcal{C}^0\left(\overline{\Omega}\right)}\;e^{-\left(\beta_1\left(K-\epsilon\right)-\rho\dfrac{\epsilon}{K}\right)\;t},\\
  	\\
  	\Phi\left(t,x\right)\leq \| \Phi_0\|_{\mathcal{C}^0\left(\overline{\Omega}\right)}\;e^{-\left(\beta_2\left(K-\epsilon\right)-\gamma\dfrac{\epsilon}{K}\right)\;t},
  	\end{array}$$
  	for all $\left(t,x\right)\in\left[0,+\infty\right)\times\overline{\Omega}$. In particular, if $\rho\dfrac{\epsilon}{K}-\beta_1\left(K-\epsilon\right)<0$ and $\gamma\dfrac{\epsilon}{K}-\beta_2\left(K-\epsilon\right)<0$ we get that $T\left(t,x\right),\Phi\left(t,x\right)\to 0$ uniformly in $x$ as $t\to +\infty$. Finally, there exists $N_{\max}$ such that $N\left(t,x\right)\leq N_{\max}$ for all $\left(t,x\right)\in\left[0,+\infty\right)\times\overline{\Omega}$.
  	
%
   \end{lem}

\begin{proof}
	
  Since $N$ is increasing, we get
  $$N\left(t,x\right)\geq N_0\left(x\right)> K-\epsilon\quad\forall t\geq0,\;\;\forall x\in\overline{\Omega}.$$
  
  Using now that $T,\;\Phi\geq 0$, 
  
  $$1-\dfrac{T+N+\Phi}{K}\leq 1-\dfrac{N}{K}<1-\dfrac{K-\epsilon}{K}=\dfrac{\epsilon}{K}.$$
  
  Therefore, $T$ satisfies
  
  $$\dfrac{\partial T}{\partial t}-\Delta\;T=f_1\left(T,N,\Phi\right)\leq\rho\;T\;\dfrac{\epsilon}{K}-\beta_1\;\left(K-\epsilon\right)\;T=\left(\rho\;\dfrac{\epsilon}{K}-\beta_1\left(K-\epsilon\right)\right)T.$$
  
  In particular, $T\leq S$, where $S$ is the unique solution of the following problem
   
    \begin{equation}\label{acotacionT}
   \left\{\begin{array}{l}
   \dfrac{\partial S}{\partial t}-\Delta\;S= -\left(\beta_1\left(K-\epsilon\right)-\rho\;\dfrac{\epsilon}{K}\right)S\;\;\text{  in  }\left[0,+\infty\right)\times\overline{\Omega},\\
   \\
   S\left(0,x\right)=\| T_0\|_{\mathcal{C}^0\left(\overline{\Omega}\right)}\text{  in  }\Omega,\\
   \\
   \dfrac{\partial S}{\partial n}\Bigg\vert_{\partial\Omega}=0.
   \end{array}
   \right.
   \end{equation} 
   
   Solving $\left(\ref{acotacionT}\right)$ we conclude that
    
   $$T\left(t,x\right)\le S\left(t,x\right)=\| T_0\|_{\mathcal{C}^0\left(\overline{\Omega}\right)}\;e^{-\left(\beta_1\left(K-\epsilon\right)-\rho\dfrac{\epsilon}{K}\right)\;t}.$$ 
   	
   	Therefore, if $\left(\rho\;\dfrac{\epsilon}{K}-\beta_1\left(K-\epsilon\right)\right)<0$, then, $S\left(t,x\right)$ satisfies that
   
   \begin{equation}\label{acotacion_T2}
   T\left(t,x\right)\leq S\left(t,x\right)=\| T_0\|_{\mathcal{C}^0\left(\overline{\Omega}\right)}\;e^{-\left(\beta_1\left(K-\epsilon\right)-\rho\dfrac{\epsilon}{K}\right)\;t}\rightarrow0
     \end{equation}
   uniformly for $x\in\overline{\Omega}$ as $t\rightarrow+\infty$.
   \\
   
   On the other hand, since $T\leq K$ and $N\ge K-\epsilon$, we get

   $$f_3\left(T,N,\Phi\right)\leq\dfrac{\gamma}{K}\;T\;\Phi\;\dfrac{\epsilon}{K}-\beta_2\left(K-\epsilon\right)\;\Phi\leq\left(\gamma\dfrac{\epsilon}{K}-\beta_2\left(K-\epsilon\right)\right)\Phi.$$
   
   Hence, we deduce that $\Phi$ satisfies

   \begin{equation}\label{acotacionF}
   \left\{\begin{array}{l}
   \dfrac{\partial \Phi}{\partial t}\leq-\left(\beta_2\left(K-\epsilon\right)-\gamma\dfrac{\epsilon}{K}\right)\Phi\text{  in  }\left[0,+\infty\right)\times\overline{\Omega},\\
   \\
   \Phi\left(0,x\right)\leq\| \Phi_0\|_{\mathcal{C}^0\left(\overline{\Omega}\right)}\text{  in  }\overline{\Omega}.
   \end{array}
   \right.
   \end{equation} 
   
   As consequence, if $\left(\gamma\dfrac{\epsilon}{K}-\beta_2\left(K-\epsilon\right)\right)<0$,
   
   \begin{equation}\label{acotacionFi2}
   \Phi\left(t,x\right)\leq \| \Phi_0\|_{\mathcal{C}^0\left(\overline{\Omega}\right)}\;e^{-\left(\beta_2\left(K-\epsilon\right)-\gamma\dfrac{\epsilon}{K}\right)t}\rightarrow0
   \end{equation}
    as $t\rightarrow+\infty$ uniformly for $x\in\overline{\Omega}$.
  \\

%
%
   
  Finally, we can obtain an uniform upper bound in time and space for $N\left(t,x\right)$ using the same argument that in $\left(\ref{eqN_t}\right)$. Now, using the upper bound for $T\left(t,x\right)$ and $\Phi\left(t,x\right)$ given in $\left(\ref{acotacion_T2}\right)$ and $\left(\ref{acotacionFi2}\right)$ one has,
   
   $$a\left(t,x\right)\leq \widehat{a}\left(t\right)=\beta_2\;\| \Phi_0\|_{\mathcal{C}^0\left(\overline{\Omega}\right)}e^{-\left(\beta_2\left(K-\epsilon\right)-\gamma\dfrac{\epsilon}{K}\right)t}+\beta_1\;\| T_0\|_{\mathcal{C}^0\left(\overline{\Omega}\right)}e^{-\left(\beta_1\left(K-\epsilon\right)-\rho\dfrac{\epsilon}{K}\right)\;t},$$
   hence,
   $$\int_{0}^{t}a\left(s,x\right)\;ds\leq\widehat{A}\left(t\right)=\int_{0}^{t}\widehat{a}\left(s\right)\;ds\le\mathcal{C}_1,$$
   and
   $$\begin{array}{ll}
   b\left(t,x\right)\leq\widehat{b}\left(t\right)=&\alpha\| T_0\|_{\mathcal{C}^0\left(\overline{\Omega}\right)}e^{-\left(\beta_1\left(K-\epsilon\right)-\rho\dfrac{\epsilon}{K}\right)\;t}+\\
   \\
   &+\delta\| \Phi_0\|_{\mathcal{C}^0\left(\overline{\Omega}\right)}\;e^{-\left(\beta_2\left(K-\epsilon\right)-\gamma\dfrac{\epsilon}{K}\right)\;t}\| T_0\|_{\mathcal{C}^0\left(\overline{\Omega}\right)}\;e^{-\left(\beta_1\left(K-\epsilon\right)-\rho\dfrac{\epsilon}{K}\right)t},
   \end{array}$$
   hence,
   $$\int_0^t b(s,x) e^{-A(s,x)}\le\mathcal{C}_2.$$
   
   With a similar reasoning to the used in Lemmas \ref{lemaestabilidad1} and \ref{lema_autovalor} we conclude the existence of $N_*\in L^\infty(\Omega)$ and $N_{\max}>0$ such that $N(t,x)\to N_*(x)\leq N_{\max}$ pointwisely in space when $t\rightarrow+\infty$.
\end{proof}

\begin{col}
Assume hypotheses of Lemma $\ref{lemaestabilidad2}$ and given $N_*\in \mathcal{C}^0\left(\overline{\Omega}\right)$ such that $N_*(x)\ge N_*^{\min}\ge K-\epsilon$ with $\epsilon>0$ small enough, then the semi-trivial steady solution $(0,N_*(x),0)$ is locally stable in $C^0(\overline\Omega)$
\end{col}

\begin{proof}
	Using the argument of Lemma $\ref{lemaestabilidad2}$, it is similar to the proof of Corollary $\ref{corollary_estabilidad_1}$.
\end{proof}

\subsection{Numerical Simulations}
In order to see the asymptotic behaviour of problem $\left(\ref{probOriginal}\right)$ 
with the boundary condition $\left(\ref{condifronte}\right)$
graphically, we will show $3$ numerical simulations for different initial conditions in the domain $\Omega=\left(-2,2\right)^2$. In all of them, we have considered the constant initial vasculature $\Phi_0\left(x\right)=0.5$ and initial necrosis zero, $N_0\left(x\right)=0$ for all $x\in\overline{\Omega}$. The parameters are taken as:

%
%

\begin{table}[H]
	\centering
	\begin{tabular}{c|c|c|c|c|c|c|c}
		
		Parameter&	$\rho$	& $\alpha$  &  $\beta_1$ & $\beta_2$ & $\gamma$ &$\delta$ & $K$   \\
		\hline
		Value&	$1$ & $0.03$ & $0.03$  & $0.03$ & $0.003$ & $0.3$&$1$  \\
		
	\end{tabular}
	\caption{\label{Table2} Parameters value.}
\end{table}

Note that hypothesis $\left(\ref{CondiNlejosKdelta}\right)$ is satisfied, hence tumor and vasculature will vanish at infinity time. Indeed, starting with the different initial conditions for the tumor given in Figure $\ref{tumor_inicial}$:

\begin{figure}[H]
	\centering
	\begin{subfigure}[b]{0.25\linewidth}
		\includegraphics[width=1.2\linewidth]{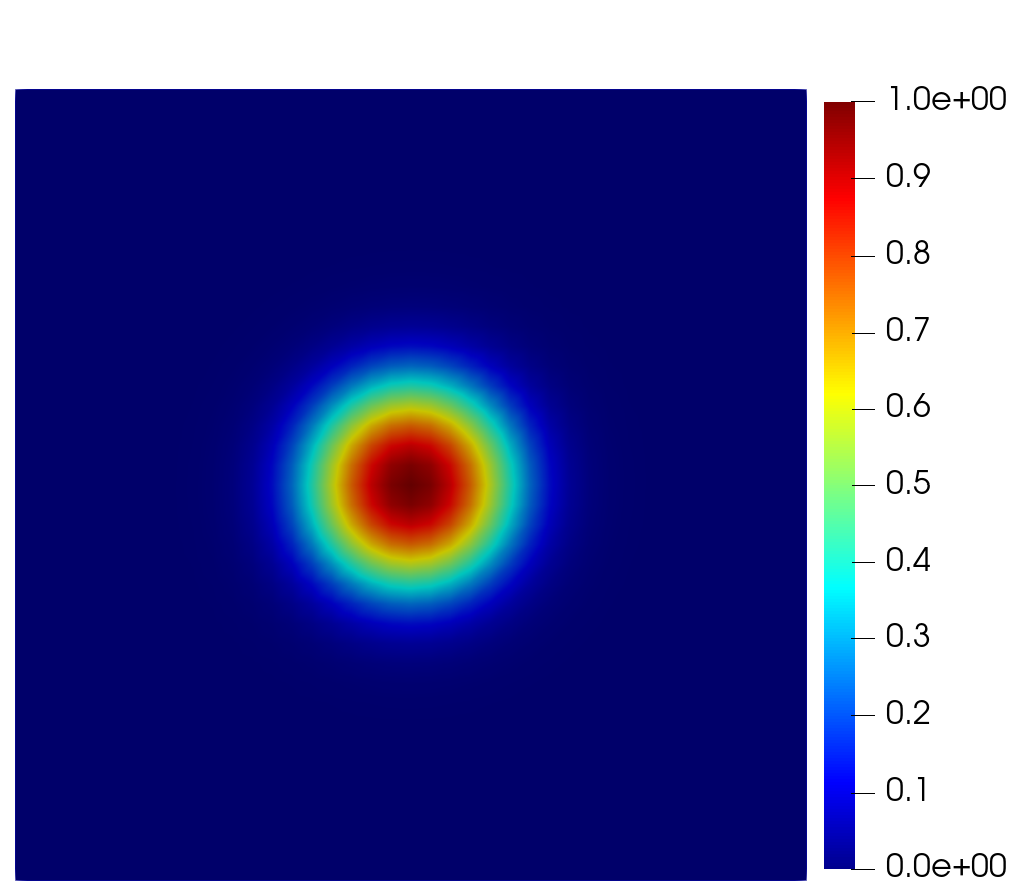}
		\centering
			\caption{ One tumor.}
	\end{subfigure}
	\hspace{1cm}
	\begin{subfigure}[b]{0.25\linewidth}
		\includegraphics[width=1.2\linewidth]{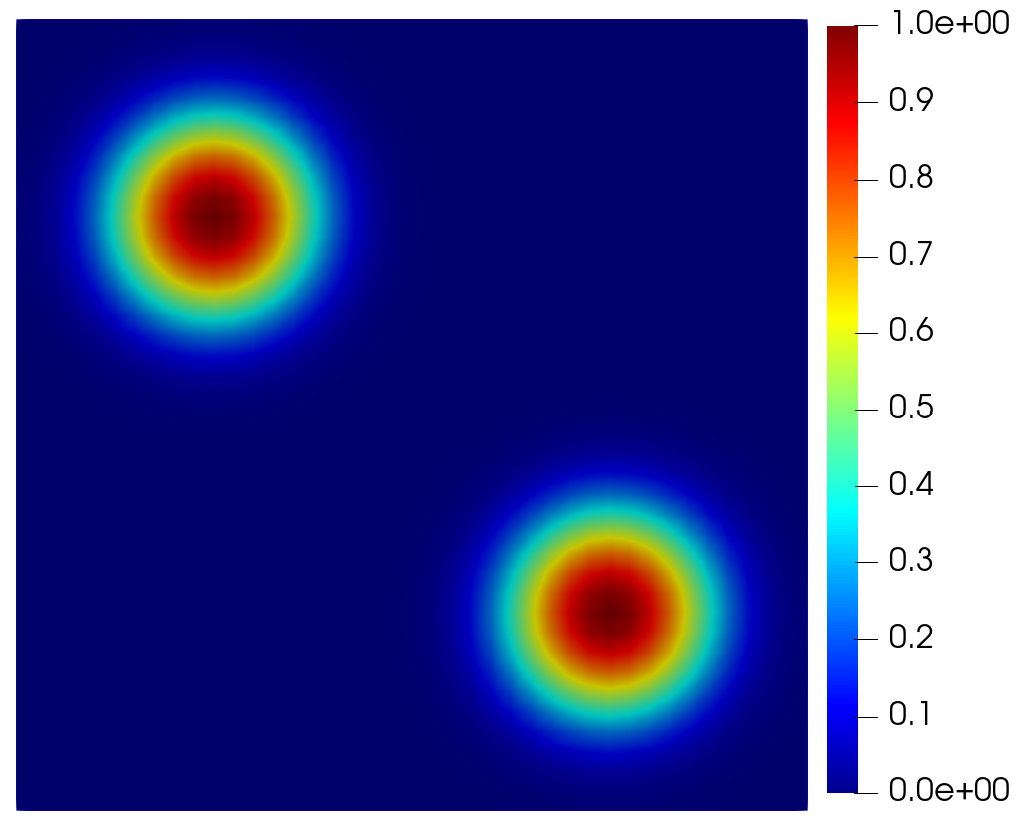}
		\centering
			\caption{ Two tumors.}
	\end{subfigure}
	\hspace{1cm}
	\begin{subfigure}[b]{0.25\linewidth}
		\includegraphics[width=1.2\linewidth]{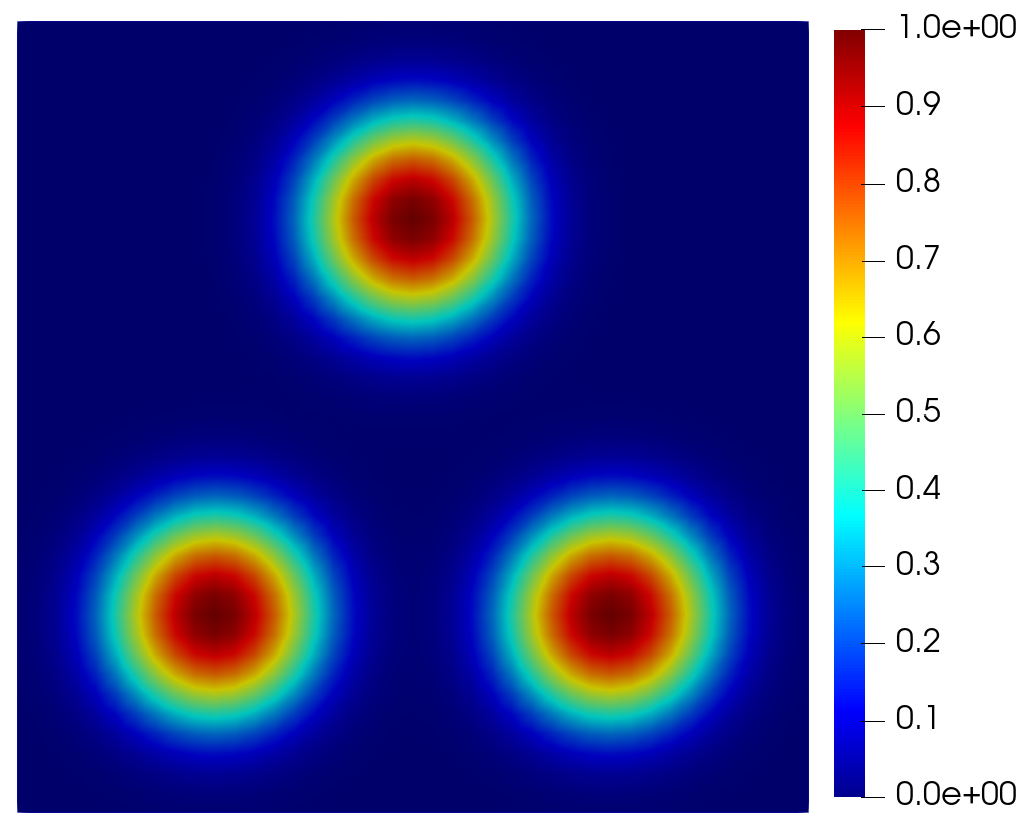}
		\centering
		\caption{ Three tumors.}
	\end{subfigure}

	\caption{ Initial tumor.}
	\label{tumor_inicial}
\end{figure}

we obtain the different equilibrium solutions for the necrosis given in Figure $\ref{necrosis_final}$:

\begin{figure}[H]
	\centering
	\begin{subfigure}[b]{0.25\linewidth}
		\includegraphics[width=1.2\linewidth]{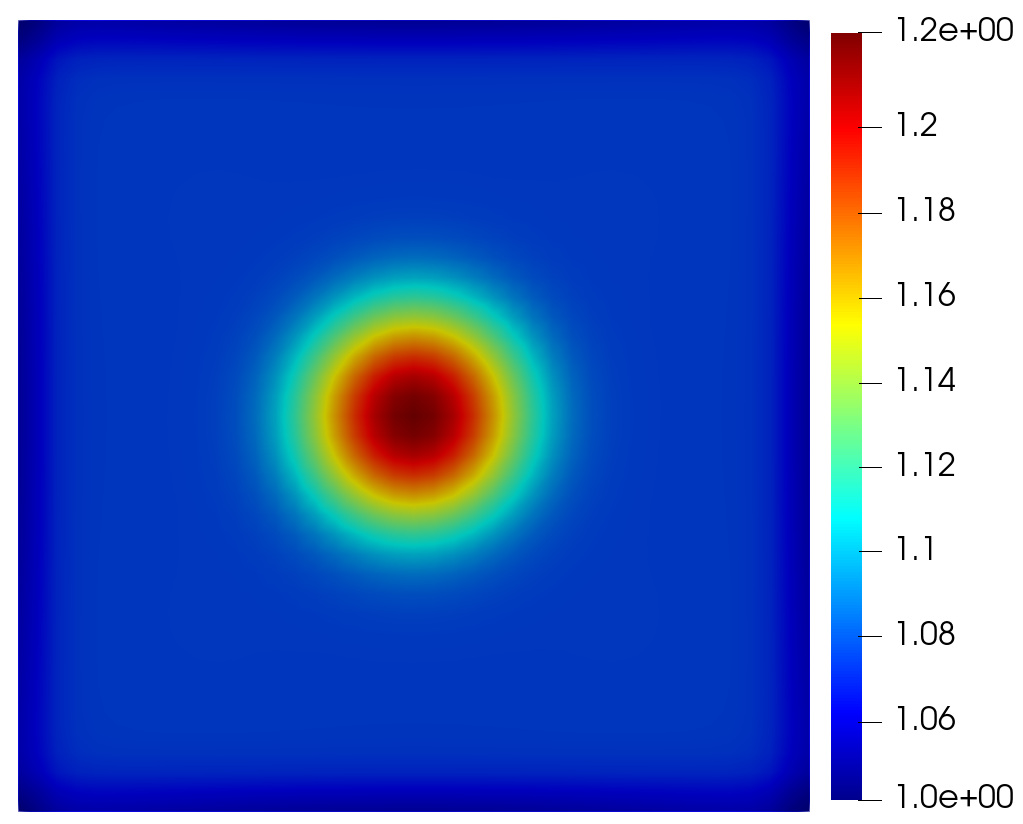}
		\centering
			\caption{One tumor.}
	\end{subfigure}
	\hspace{1cm}
	\begin{subfigure}[b]{0.25\linewidth}
		\includegraphics[width=1.2\linewidth]{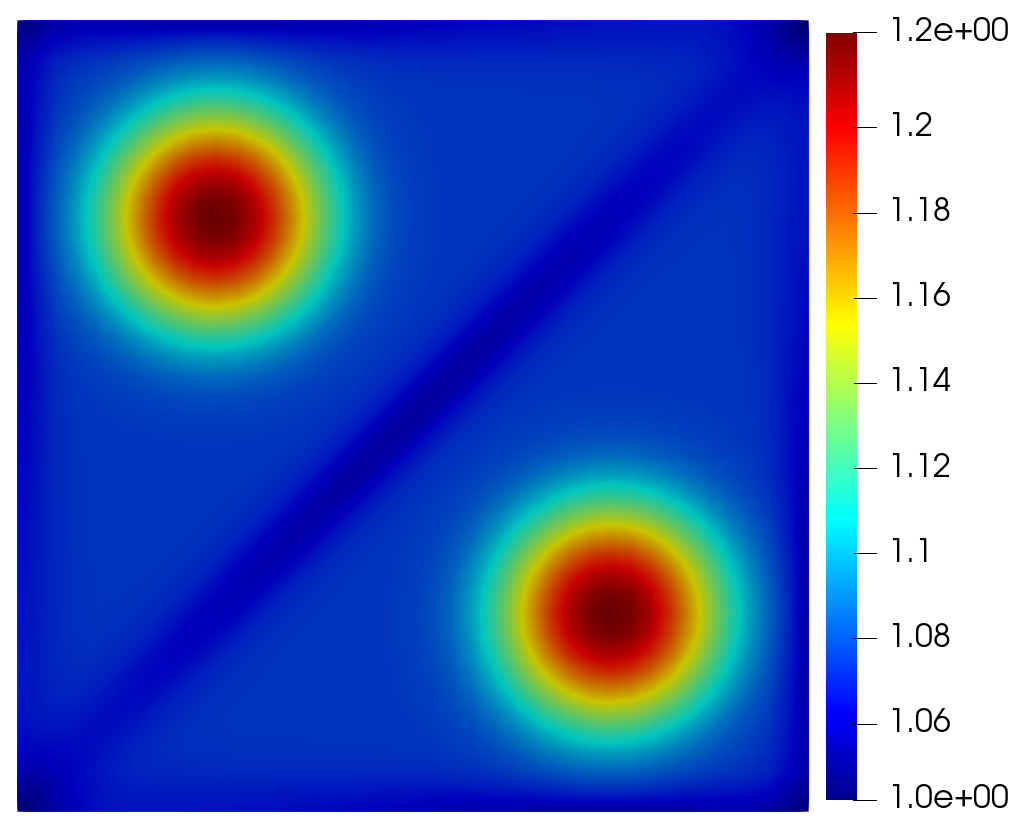}
		\centering
		\caption{Two tumors.}
	\end{subfigure}
	\hspace{1cm}
	\begin{subfigure}[b]{0.25\linewidth}
		\includegraphics[width=1.2\linewidth]{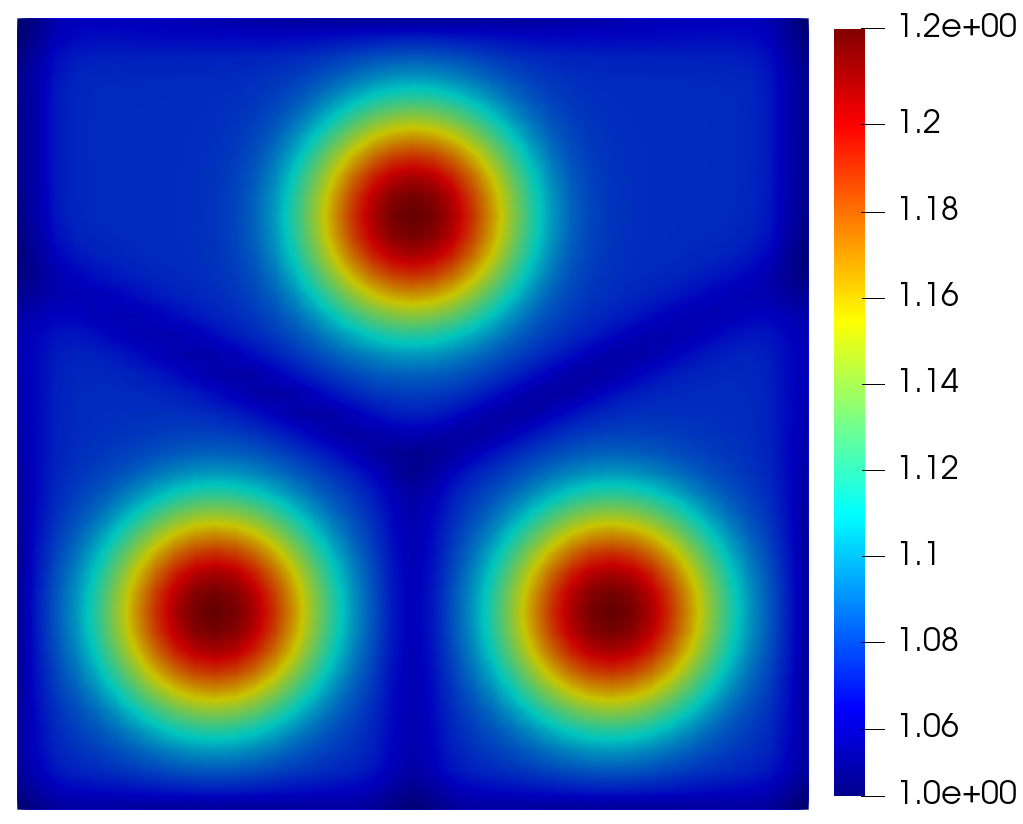}
		\centering
		\caption{Three tumors.}
	\end{subfigure}
	
	\caption{Final necrosis.}
	\label{necrosis_final}
\end{figure}

We observe as necrosis occupies all the domain but mainly where the tumor was initially, and tumor and vasculature disappear in all the cases. Moreover, the maximum value of necrosis is the same in each simulation. 
\\

In order to see the importance of hypothesis $\left(\ref{CondiNlejosKdelta}\right)$, now we consider the value of the parameters as follows:

%
%

\begin{table}[H]
	\centering
	\begin{tabular}{c|c|c|c|c|c|c|c}
		
		Parameter&	$\rho$	& $\alpha$  &  $\beta_1$ & $\beta_2$ & $\gamma$ &$\delta$ & $K$   \\
		\hline
	Value&	$1$ &$0.03$ & $0.03$  & $0.03$ & $0.3$ & $0.03$&$1$  \\
		
	\end{tabular}
	\caption{\label{Table3}Parameters value.}
\end{table}

where $\left(\ref{CondiNlejosKdelta}\right)$ is not satisfied. Then, starting with the same initial conditions for the tumor given in Figure $\ref{tumor_inicial}$, we obtain the different equilibrium solutions for the necrosis given in Figure $\ref{necrosis_final2}$:

\begin{figure}[H]
	\centering
	\begin{subfigure}[b]{0.25\linewidth}
		\includegraphics[width=1.2\linewidth]{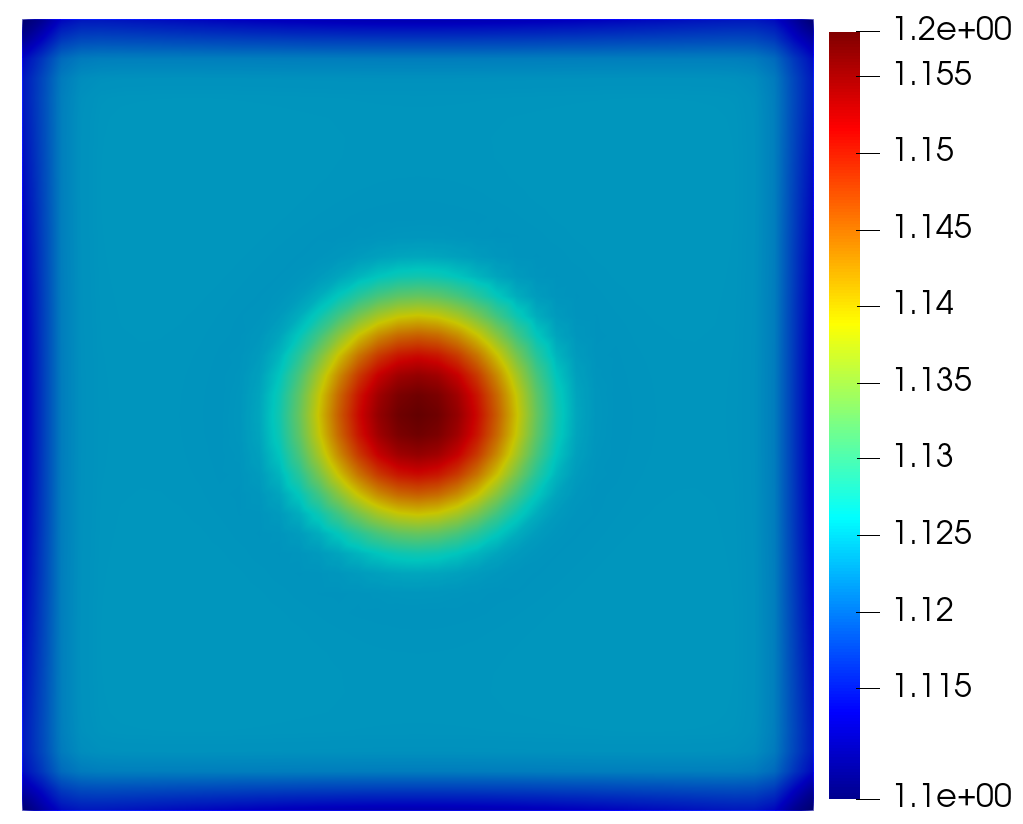}
		\centering
		\caption{One tumor.}
	\end{subfigure}
	\hspace{1cm}
	\begin{subfigure}[b]{0.25\linewidth}
		\includegraphics[width=1.2\linewidth]{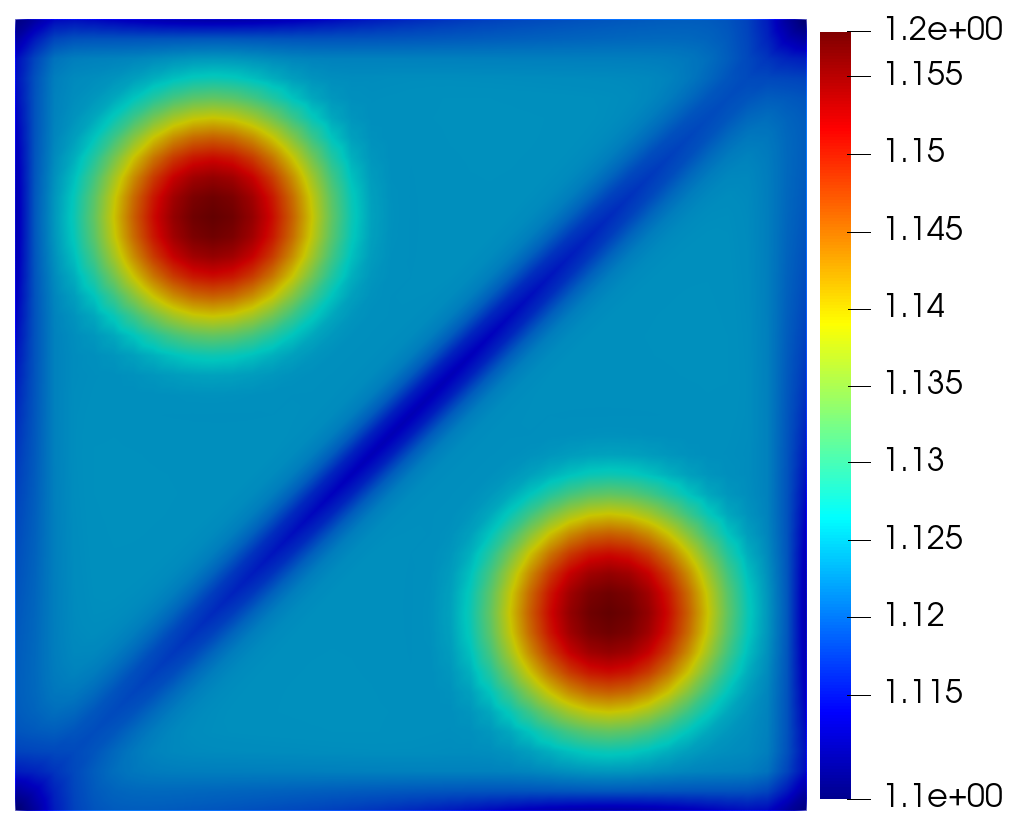}
		\centering
		\caption{Two tumors.}
	\end{subfigure}
	\hspace{1cm}
	\begin{subfigure}[b]{0.25\linewidth}
		\includegraphics[width=1.2\linewidth]{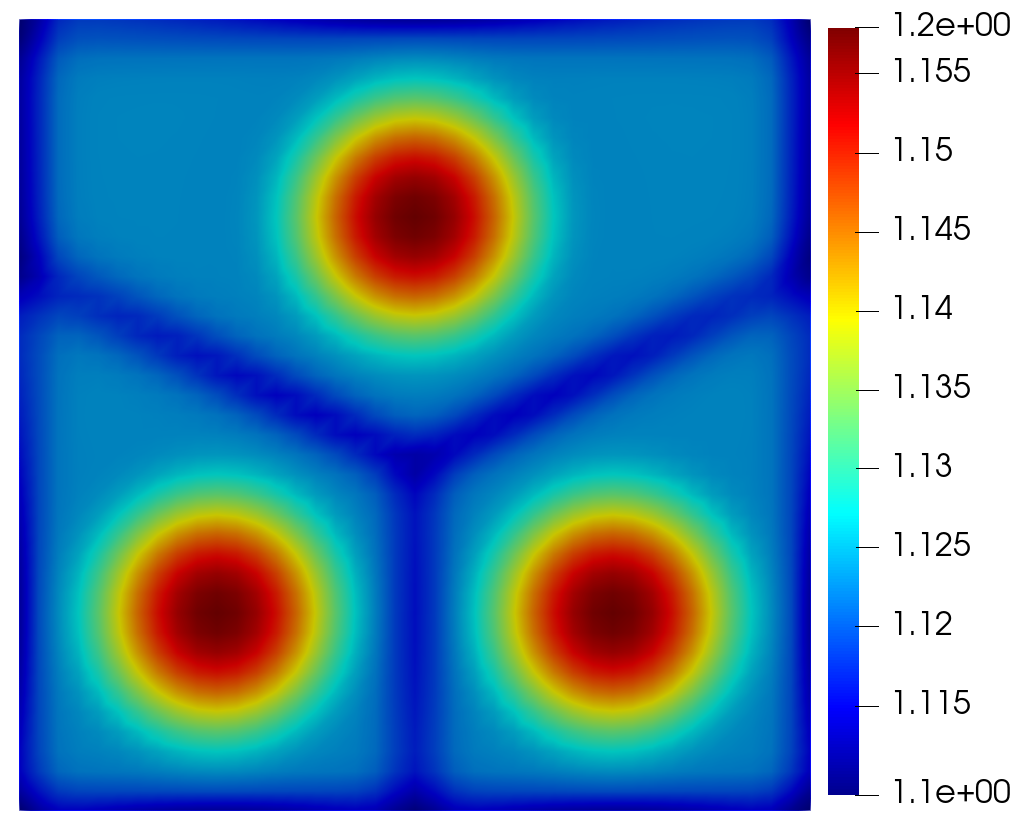}
		\centering
		\caption{Three tumors.}
	\end{subfigure}
	
	\caption{Final necrosis.}
	\label{necrosis_final2}
\end{figure}

We observe a similar behaviour that in Figure $\ref{necrosis_final}$.



%
\section{Conclusion}\label{conclusion}
One of the types of cancers that has attracted the most attention is glioblastoma. In this paper we have completed the model studied in \cite{Victor_2020} introducing a new variable, the vasculature, giving rise a more realistic model. Thus, after the theoretical and numerical study made of $\left(\ref{probOriginal}\right)$-$\left(\ref{condinicio}\right)$, we can conclude:

\begin{enumerate}
	\item The model $\left(\ref{probOriginal}\right)$ is well-posed: we have proved that there exists a unique global in time classical solution of the model.

	\item The long time behaviour of $\left(\ref{probOriginal}\right)$-$\left(\ref{condinicio}\right)$ asserts that vasculature always disappears and, under conditions on the parameters of the problem, tumor proliferative also disappears. Moreover, we show numerical simulations that highlight our results. 
\end{enumerate}


\begin{thebibliography}{9}
	\bibitem{Abrol_2017} 
	S. Abrol, A. Kotrotsou, A. Salem, P. O. Zinn, R. R. Colen.: Radiomic phenotyping in brain cancer to unravel hidden information in medical images. Top in Magn Reson Imaging. 26, 43-53 (2017)
	
	\bibitem{Amann_1991}
	H. Amann.: Highly degenerate quasilinear parabolic systems. Ann. Scuola Norm. Sup. Pisa Cl. Sci. Ser. 4, 18, 135-166 (1991)
	
	\bibitem{Amann_1993}
	H. Amann.: Nonhomogeneous linear and quasilinear elliptic and parabolic boundary value problems. Vieweg+Teubner Verlag, Wiesbaden (1993)
	
	\bibitem{Bitsoumi_2017}
	V. Bitsouni, M. A. J. Chaplain, R. Eftimie.:Mathematical modelling of cancer invasion: the multiple roles of \uppercase{T}\uppercase{G}\uppercase{F}-$\beta$ pathway on tumour proliferation and cell adhesion. Math. Models Methods Appl. Sci. 27, 1929-1962 (2017)
	
	\bibitem{Wiley_2003}
	R. S. Cantrell, C. Cosner.: Spatial ecology via reaction-diffusion equations. John Wiley \& Sons, Ltd, Chichester, West Sussex, England; Hoboken, NJ (2003)
	
	\bibitem{Coddington_1955}
	E. A. Coddington, N. Levinson.: Theory of ordinary differential equations. McGraw-Hill Book Company, Inc., New York-Toronto-London (1955)
	
	\bibitem{Cruz_2018}
	E. Cruz, M. Negreanu, J. I. Tello.: Asymptotic behavior and global existence of solutions to a two-species chemotaxis system with two chemicals. Z. Angew. Math. Phys. 69, 107 (2018)
	
	\bibitem{Anderson_2015}
	A. L. de Araujo, Paulo M.D. de Magalhães.: Existence of solutions and optimal control for a model of tissue invasion by solid tumours. J. Math. Anal. Appl. 421, 842-887 (2015)
	
	\bibitem{Deuel_1978}
	J. Deuel, P. Hess.: Nonlinear parabolic boundary value problems with upper and lower solutions. Israel J. Math. 29, 92-104 (1978)
	
	\bibitem{Ellingson_2014}
	B. M. Ellingson.: Radiogenomics and imaging phenotypes in glioblastoma: novel observations and correlation with molecular characteristics. Curr Neurol Neurosci Rep. 15, 506 (2014)
	
	\bibitem{Feireisl_2009}
	E. Feireisl, A. Novotny. Singular limits in thermodynamics of viscous fluids. Advances in Mathematical Fluid Mechanics, Birkhäuser Verlag, Basel (2009)
	
	\bibitem{Gillies_2016}
	R. J. Gillies, P. E. Kinahan, H. Hricak.: Radiomics: images are more than pictures, they are data. Radiology. 278, 563-577 (2016)
	
	\bibitem{Klank_2018}
	R. L. Klank, S. S. Rosenfeld, D. J. Odde.: A brownian dynamics tumor progression simulator with application to glioblastoma. Converg. Sci. Phys. Oncol. 4 015001 (2018)
	
	\bibitem{Kubo_2016}
	A. Kubo, J.I. Tello.: Mathematical analysis of a model of chemotaxis with competition terms. Differential Integral Equations. 29, 441-454 (2016)
	
	\bibitem{Lions_1969}
	J-L. Lions.: Quelques métodes de résolution des problémes aux limites non linéares. Durnod, Gauthier-Villars, Paris (1969)
	
	\bibitem{Alicia_2015}
	A. Mart{\'{\i}}nez-Gonz{\'{a}}lez, M. Dur{\'{a}}n-Prado, G. F. Calvo, F. J. Alca{\'{\i}}n, L. A. P{\'{e}}rez-Romasanta, V. M. P{\'{e}}rez-Garc{\'{\i}}a.: Combined therapies of antithrombotics and antioxidants delayin silicobrain tumour progression. Mathematical Medicine and Biology. 32, 239-262 (2015)
	
	\bibitem{Alicia_2012}
	A. Martínez-González, G. F. Calvo, L. A. Pérez-Romasanta, V. M. Pérez-García.: Hypoxic cell waves around necrotic cores in glioblastoma: a biomathematical model and its therapeutic implications. Bulletin of Mathematical Biology. 74, 2875-2896 (2012)
	
	\bibitem{Narang_2016}
	S. Narang, M. Lehrer, D. Yang, J. Lee, A. Rao.: Radiomics in glioblastoma: current status, challenges and potential opportunities. Translational Cancer Research. 5, 383-397 (2016)
	
	\bibitem{Ostrom_2014}
	Q. T. Ostrom, H. Gittleman, P. Liao, C. Rouse, Y. Chen, J. Dowling, Y. Wolinsky, C. Kruchko, J. Barnholtz-Sloan.: \uppercase{C}\uppercase{B}\uppercase{T}\uppercase{R}\uppercase{U}\uppercase{S} statistical report: primary brain and central nervous system tumors diagnosed in the united states in 2007-2011. Neuro-Oncology. 16, iv1-iv63 (2014)
	
	\bibitem{Pang_2018}
	P. Y. H. Pang, Y. Wang.: Global boundedness of solutions to a chemotaxis-haptotaxis model with tissue remodeling. Math. Models Methods Appl. Sci. 28, 2211-2235 (2018)
	
	\bibitem{Victor_2020}
	J. P\'erez-Beteta, J. Belmonte-Beitia, V. M. P\'erez-Garc\'{\i}a.: Tumor width on T1-weighted MRI images of glioblastoma as a prognostic biomarker: a mathematical model. Math. Model. Nat. Phenom. 15, 10 (2020)
	
	\bibitem{Julian_2016}
	JJ. P{\'{e}}rez-Beteta, A. Mart{\'{\i}}nez-Gonz{\'{a}}lez, D. Molina, M. Amo-Salas, B. Luque, E. Arregui, M. Calvo, J. M. Borr{\'{a}}s, C. L{\'{o}}pez, M. Claramonte, J. A. Barcia, L. Iglesias, J. Avecillas, D. Albillo, M. Navarro, J. M. Villanueva, J. C. Paniagua, J. Martino, C. Vel{\'{a}}squez, B. Asenjo, M. Benavides, I. Herruzo, M. del Carmen Delgado, A. del Valle, A. Falkov, P. Schucht, E. Arana, L. P{\'{e}}rez-Romasanta, V. M. P{\'{e}}rez-Garc{\'{\i}}a.: Glioblastoma: does the pre-treatment geometry matter? \uppercase{A} postcontrast \uppercase{T}1 \uppercase{M}\uppercase{R}\uppercase{I}-based study. European Radiology. 27, 1096-1104 (2016)
	
	\bibitem{Victor_2018}
	J. Pérez-Beteta, D. Molina-García, J. A. Ortiz-Alhambra, A. Fernández-Romero, B. Luque, E. Arregui, M. Calvo, J. M. Borrás, B. Melédez, {\'{A}}. Rodríguez de Lope, R. Moreno de la Presa, L. Iglesias Bayo, J. A. Barcia, J. Martino, C. Velásquez, B. Asenjo, M. Benavides, I. Herruzo, A. Revert, E. Arana, V. M.  Pérez-García.: Tumor \uppercase{S}urface \uppercase{R}egularity at \uppercase{M}\uppercase{R} \uppercase{I}maging \uppercase{P}redicts \uppercase{S}urvival and \uppercase{R}esponse to \uppercase{S}urgery in \uppercase{P}atients with \uppercase{G}lioblastoma. Radiology. 288, 218-225 (2018)
	
	\bibitem{Rothe_1984}
	F. Rothe.: Global solutions of reaction‐diffusion systems. Lecture notes in mathematics, vol. 1072. Springer-Verlag, Berlin(West)‐Heidelberg‐New York‐Tokyo (1984)
	
	\bibitem{Simon_1987}
	J. Simon.: Compact sets in the space $\uppercase{L}^p\left(0,\uppercase{T},\uppercase{B}\right)$. Ann. Mat. Pura Appl. 146, 65-96 (1987)
	
	\bibitem{Tello_2018}
	J. I. Tello, D. Wrzosek.: Inter-species competition and chemorepulsion. J. Math. Anal. Appl. 459, 1233-1250 (2018)
	
	
	
	
\end{thebibliography}

%

\end{document}